\documentclass[11pt,english]{amsart}
\usepackage{amssymb,amsmath,mathabx, amscd}
\usepackage[pdftex]{graphicx}
\usepackage{enumerate}
\usepackage{tikz, tikz-cd}
\usepackage[margin=1in]{geometry}
\usepackage{color}
\usepackage{hyperref}
\usepackage{mathrsfs, colonequals, wasysym}
\usepackage{commath}
\usepackage{array}
\usepackage{spalign}
\usepackage{bbold}
\usepackage{algorithm}
\usepackage{algpseudocode}
\usepackage{xpatch}
\usepackage{etoolbox}

\newcommand{\ee}{\mathbb{E}}
\newcommand{\nn}{\mathbb{N}}

\newcommand{\zz}{\mathbb{Z}}

\newcommand{\cE}{\mathcal{E}}
\newcommand{\cF}{\mathcal{F}}

\newcommand{\eps}{\varepsilon}

\newtheorem{thm}{Theorem}[section]

\newtheorem{lem}[thm]{Lemma}

\newtheorem{prop}[thm]{Proposition}
\newtheorem{cor}[thm]{Corollary}

\xpatchcmd{\proof}{\itshape #1}{%
  \itshape\color{blue}%
  \ifstrequal{#1}{\proofname}%
    {\proofname}
    {\proofname\ #1}
}{}{}
\makeatletter
\newcommand*{\rom}[1]{\expandafter\@slowromancap\romannumeral #1@}
\makeatother

\theoremstyle{definition}
\newtheorem{defin}[thm]{Definition}

\newtheorem{example}[thm]{Example}
\newtheorem{assumption}[thm]{Assumption}

\theoremstyle{remark}
\newtheorem{rem}{Remark}

\renewcommand{\bar}{\overline}

\DeclareMathOperator*{\argmax}{\arg\!\max}

\title{Robust Reconstruction on Trees with a Growing Alphabet}
\author{Heon Lee}
\email{heonl@stanford.edu}
\address{Stanford University, Department of Statistics}

\begin{document}

\begin{abstract}
We study robust reconstruction for the \(q\)-state Potts broadcast process on a Galton-Watson tree in the growing-alphabet regime \(q\to\infty\), with boundary depth growing sufficiently quickly in relation to \(q\). We show that, in the regime \(d\lambda >1\), where \(d\) is the expected number of offspring and \(\lambda\) is the nontrivial eigenvalue of the Potts transition matrix, the root posterior is asymptotically unchanged by a broad class of noise channels applied independently to the boundary labels. This remains true even when the probability of retaining the true label at an individual boundary vertex tends to zero as \(q\to\infty\). More precisely, under suitable moment and noise assumptions, the noiseless and noisy root posteriors converge to one another in expected total variation, and hence have the same asymptotic Bayes-optimal reconstruction accuracy.
\end{abstract}

\maketitle

\tableofcontents

\section{Introduction}

Broadcasting on a tree is a canonical model for studying information propagation through a branching network. Given a rooted tree, a label is assigned to the root and propagated independently along the edges through a fixed Markov channel, \(M\). This topic has been studied in information theory, mathematical genetics, and statistical physics (see, e.g.,~\cite{EKPS00, MP03}). We consider the ferromagnetic Potts broadcast process on Galton-Watson trees with mean offspring \(d\), whose vertices carry labels in the alphabet \([q] = \{1,\dots, q\}\). The root receives a uniformly random label. Along each parent-to-child edge, the child vertex copies the parent's label with probability \(\lambda\in(0,1)\) and otherwise receives an independently generated uniform label in \([q]\). We call an edge copying the parent's label a copy edge. Otherwise, call the edge mutated. Not every vertex that shares its parent's label is the endpoint of a copy edge, since the label along a mutated edge may coincide with the parent label by random chance.

The central questions are how well one can estimate the root label from the boundary as the tree depth grows, whether nontrivial reconstruction remains possible asymptotically in \(q\), and where the sharp reconstruction threshold lies. Here, the boundary refers to the depth-\(k_q\) vertices, where \(k_q\) diverges as \(q\to\infty\). Hence, not all leaf vertices are boundary vertices.

A related line of work studies reconstruction from boundary labels that are observed with noise. One may ask whether nontrivial reconstruction remains possible from the noisy boundary. A stronger question is whether the additional noise changes the Bayes-optimal reconstruction accuracy. This is the question we study in the growing-alphabet regime: \begin{center}
\begin{minipage}{0.82\textwidth}
\itshape
As the size of the alphabet and tree depth jointly diverge, can noisy boundary observations yield the same asymptotic Bayes-optimal reconstruction accuracy as noiseless observations, even when individual observations become increasingly unreliable?
\end{minipage}
\end{center}

When the alphabet is large, repeated labels are unlikely to arise from independent redraws along mutated edges. Repetitions in different child-subtrees of a vertex, meaning subtrees rooted at distinct children of that vertex, therefore provide evidence about ancestral labels. The main step is to show that these identified labels retain almost all the information needed for optimal reconstruction, and the root label remains recoverable from a sufficiently deep noisy boundary. We assume that errors are spread across incorrect labels, with the probability of any particular incorrect output sufficiently small relative to the minimum probability of retaining the true label. Our main result, Theorem~\ref{thm:tree-theorem}, can be informally summarized as follows. 
\begin{thm}(Informal)
Suppose $d\lambda>1$ and the offspring distribution has sufficiently many finite moments. As $q\to\infty$, the root posteriors based on noisy and noiseless boundaries converge to one another in expected total variation, provided the depth grows sufficiently quickly relative to \(q\) and errors are sufficiently spread among incorrect labels relative to the retention probability. The noise may be asymmetric, and the probability of retaining the true label may tend to zero. In particular, noisy and noiseless observations have the same limiting Bayes-optimal reconstruction accuracy bounded away from zero.
\end{thm}

The threshold is \(d\lambda > 1\) because the vertices connected to the root by copy edges form a Galton-Watson tree with mean offspring \(d\lambda\), which survives with positive probability exactly when \(d\lambda>1\). 

\subsection{Related Work}

A central benchmark for the reconstruction threshold is the Kesten-Stigum (KS) criterion \(d\lambda^2 = 1\), where \(d\) is the mean number of offspring and \(\lambda\) is the nontrivial eigenvalue of \(M\) (see, e.g.,~\cite{KS66, MP03}). For the binary symmetric broadcast process on regular trees, the KS bound is the sharp reconstruction threshold~\cite{BRZ95}. The threshold was subsequently extended in~\cite{EKPS00} to arbitrary infinite trees, where \(d\) is replaced by the branching number \(\mathrm{br}(T)\) defined for an infinite tree \(T\). In the case of Galton-Watson trees, which is our primary interest, \(\mathrm{br}(T) = d\) almost surely on the event that the tree does not go extinct at some finite depth~\cite{Lyons90}. Mossel and Peres~\cite{MP03} show that the KS criterion remains the sharp threshold for larger alphabets \([q]\) on \(d\)-ary trees when we are only allowed to see the census information, i.e., the number of vertices with label \(i\) in the boundary for every \(i\in[q]\).

Larger alphabets permit qualitatively different behavior. On a regular tree, if \(d\lambda > 1\), reconstruction is possible for all sufficiently large \(q\), while if \(d\lambda \leq 1\), reconstruction is impossible as shown in~\cite{Mossel01}. In fact, for regular trees, Sly~\cite{Sly11} showed that the KS bound is not sharp for the Potts broadcast model with \(q\geq 5\), while being sharp for \(q=3\) for sufficiently large \(d\). The result was extended in~\cite{MSS23} to a broad class of Galton-Watson trees, where the KS bound is not sharp when \(q\geq 5\) but is sharp for \(q=3,4\) for sufficiently large \(d\).

On the robust reconstruction side, Janson and Mossel~\cite{JM04} showed that, on trees with a fixed alphabet, nontrivial reconstruction with noisy labels can be attained asymptotically when \(\mathrm{br}(T)\lambda^2 > 1\), which corresponds to the KS threshold for \(d\)-ary and Galton-Watson trees. The noise model corresponds to independent boundary noise that retains each label with probability \(1-\eps\) and otherwise replaces it by an independent draw from a fixed nondegenerate distribution, for any fixed \(\eps<1\).

Mossel, Neeman, and Sly~\cite{MNS16} showed that in the binary case with constant noise, the noiseless and noisy Bayes-optimal reconstruction accuracy coincide asymptotically when sufficiently above the KS threshold, i.e., \(d\lambda^2 > C\) for some constant \(C\). Yu and Polyanskiy~\cite{YP24} subsequently removed the high-signal condition, establishing the same noise-independence throughout the regime \(d\lambda^2>1\) provided that each noisy boundary observation retains some information about its true label. 

For multiple labels, Chin and Sly~\cite{CS20} proved coordinatewise convergence of the noisy and noiseless posterior vectors, consequently showing that the reconstruction accuracies are asymptotically equivalent, under a sufficiently strong condition \(d\lambda^2 > C(q)\) and sufficiently accurate initial observations.  Their later work~\cite{CS26} treats a class of asymmetric broadcast models. Gu and Polyanskiy~\cite{GP23} extended the result of~\cite{YP24}, obtaining Potts belief-propagation uniqueness and stability results on regular and Poisson trees, including the sufficient condition \(d\lambda^2>1+C\max\{\lambda,q^{-1}\}\log q\) for a constant \(C\) independent of \(d, \lambda\), and \(q\). Their analysis also includes an extension to possibly asymmetric full-rank initial channels. These results address optimal inference beyond the binary setting, but their stated sufficient conditions do not cover fixed \(d,\lambda\) with \(d\lambda>1\) as \(q\to\infty\).

In contrast, for fixed \(d\) and \(\lambda\), our result applies throughout \(d\lambda>1\), including parameter values below the Kesten-Stigum bound, under the moment, depth, and noise assumptions stated below. In this growing-alphabet setting, \(d\lambda=1\) is the sharp threshold for nontrivial reconstruction, extending the corresponding threshold for \(d\)-ary trees established by Mossel~\cite{Mossel01}. By Theorem~\ref{thm:tree-theorem}, when \(d\lambda > 1\), the noiseless and noisy reconstruction accuracy converges to \(E_\infty^\star\), a quantity shown to be bounded away from zero. On the other hand, when \(d\lambda \leq 1\), with probability tending to one, there is no path from the root to level \(k\) consisting entirely of copy edges. On this event, every boundary label lies beyond at least one fresh redraw and therefore carries no information about the root label. Consequently, both the noiseless and noisy posteriors converge to the prior, and asymptotically one can do no better than a random guess.

The growing-alphabet setting is also motivated by the sparse symmetric stochastic block model (SBM)~\cite{HLL83} when the number of communities, \(q_n = q\), grows with the number of vertices, \(n\). We focus here on the assortative case, corresponding to the ferromagnetic Potts broadcast model considered in this paper. For fixed \(q\), as the number of vertices tends to infinity, a local neighborhood of a uniformly chosen vertex, together with its community labels, converges to a \(q\)-state Potts broadcast process on a Poisson Galton-Watson tree~\cite{CS20, CS26, MNS15, MNS16, MSS23}. This correspondence closely links reconstruction on the tree and recovery in the SBM. For \(q=2, d\lambda^2=1\) is sharp for both tree reconstruction~\cite{BRZ95, EKPS00} and SBM weak recovery~\cite{Massouli14, MNS15, MNS18}. For \(q\geq 5\), reconstruction below the KS bound is possible on regular trees~\cite{Sly11} and a broad class of Galton-Watson trees~\cite{MSS23}. On the graph side, information-theoretic recovery below the KS bound was established independently by Abbe and Sandon~\cite{AS18} for \(q\geq5\) and Banks, Moore, Neeman, and Netrapalli~\cite{BMNN16} for \(q\geq 11\). In the growing-community regime, the relevant information-theoretic scale in both settings is instead \(d\lambda=\Theta(1)\). We show that the tree reconstruction threshold is \(d\lambda=1\), while Chin, Mossel, Sohn, and Wein~\cite{CMSW25} identify the \(d\lambda=\Theta(1)\) scale of the information-theoretic threshold for SBM recovery in the growing-community regime, proving impossibility when \(d\lambda<1\) for \(q=o(n)\) and an exponential-time algorithm succeeding when \(d\lambda>14\) for \(q=\omega(1)\).

\subsection{Main Results}\label{subsec:main-results}

Let $(T, \rho)$ denote a rooted tree. Let \(T_{\leq k} := \{v\in T: d(v, \rho)\leq k\}\) denote the subtree of vertices \(T\) with edge distance at most \(k\) from the root. Let $L_k(u)$ denote the $k$th-level descendants of $u$. 

\begin{defin}[Galton-Watson Tree]
Let \(\Phi\) be the probability generating function (PGF) of an \(\nn_0\)-valued random variable with finite first moment. A \emph{Galton-Watson tree}, \((T,\rho)\), is a random rooted tree constructed from a root \(\rho\), where each vertex \(u\), independently of all other vertices, has \(D_u\) children, with \(D_u\stackrel{\mathrm{i.i.d.}}{\sim}\Phi\).
\end{defin}

\begin{defin}[\(q\)-state Potts Broadcast Model]
The \emph{\(q\)-state Potts broadcast model} is the broadcasting of \(q\)-state spins on the vertices of a Galton-Watson tree. In particular, let \((T,\rho)\) be a Galton-Watson tree, \(q\in\nn\), and \(\lambda\in(0,1)\). Let \(\tau^{(q)}=\tau\in[q]^{|T|}\) denote the labels on the tree, where \(\tau_u\in[q]\) denotes the label of vertex \(u\in T\). The root label is generated uniformly at random in \([q]\). Conditional on the parent labels, labels are propagated independently along the edges according to the \(q\)-state Potts channel \[M^{(q)} = M=\lambda I+\frac{1-\lambda}{q}\mathbf{1}\mathbf{1}^\top.\] That is, for every parent-to-child edge \(u\to v\), \[\Pr(\tau_v=j\mid \tau_u=i)=M_{ij}.\]
\end{defin}

We note that \(\lambda\) is the nontrivial eigenvalue of \(M\), and as \(q\to\infty\), \(M_{ii}\to \lambda\).

Every vertex \(u\in L_k(\rho)\) receives a noisy label based on its true label and noise matrix \(\Delta^{(q)} = \Delta\in[0,1]^{q\times q}\). The noisy label of vertex \(u\) is denoted with \(\tilde\tau_u^{(q)} = \tilde\tau_u\), where \[\Pr(\tilde\tau_u = j| \tau_u = i) = \Delta_{ij}.\] Clearly, the noise matrix must satisfy \(\sum_{j=1}^q \Delta_{ij} = 1\) for all \(i\in [q]\). We assume that the noisy labels are generated according to the noise matrix independently for each vertex. We refer to \(\tau_u\) as the true or noiseless label and to \(\tilde{\tau}_u\) as the noisy or perturbed observation. An observation is correct if \(\tilde{\tau}_u=\tau_u\) and erroneous otherwise.

\begin{defin}[Tree Reconstruction Accuracy/Robust Tree Reconstruction Accuracy]\label{defin:tree-recon-acc}
Let \[X_{q,k}(i) := \Pr(\tau_\rho=i | T_{\leq k}, \tau_{L_k(\rho)})\] and \[W_{q,k}^{\Delta}(i) = W_{q,k}(i) := \Pr(\tau_\rho=i | T_{\leq k},\tilde\tau_{L_k(\rho)})\] be the posterior probability of the root vertex having true label \(i\in[q]\) given the true and observed, respectively, depth-\(k\) labels. Then, the maximum likelihood estimate (equivalently, Bayes' estimator) of the root label given the true and observed depth-\(k\) labels can be defined as \[\hat\tau_\rho^{(q)}(k) = \hat\tau_\rho(k) := \argmax_i X_{q, k}(i)\] and \[\hat{\tilde \tau}_\rho^{(q)}(k) = \hat{\tilde \tau}_\rho(k) := \argmax_i W_{q, k}(i),\] respectively.

The \emph{tree reconstruction accuracy} is the probability of predicting the correct label given the true labels: \[p_T(\Phi,\lambda, q, k)= \Pr(\hat\tau_\rho(k) = \tau_\rho) = \ee\left[\max_i X_{q, k}(i)\right].\] Similarly, the \emph{robust tree reconstruction accuracy} is the probability of predicting the correct label given the noisy labels: \[\tilde p_T(\Phi,\lambda, q, k, \Delta^{(q)})  = \Pr(\hat{\tilde\tau}_\rho(k) = \tau_\rho) = \ee\left[\max_i W_{q, k}(i)\right].\] 
\end{defin}

When clear from the context, we will often write \(p_T(q)\) and \(\tilde p_T(q) = \tilde p_T(q, \Delta^{(q)})\) to mean \(p_T(\Phi, \lambda, q, k_q)\) and \(\tilde p_T(\Phi,\lambda, q, k_q, \Delta^{(q)})\), respectively. 

Let \[a_q := \min_i \Delta_{ii}\text{ and }b_q :=\max_{i\neq j}\Delta_{ij}.\]

Two separate sufficient regularity conditions on the noise channel are given in Assumptions~\ref{assump:noise-matrix} and~\ref{assump:noise-matrix-2}. The two assumptions reflect a balance between the growth of the inherited signal from observed labels matching the true root label and the accumulation of erroneous observations. 

Suppose \(d\lambda > 1\), and let \(\zeta := \frac{\log d}{\log(d\lambda)}\). Consider the depth-\(k\) descendants of the root. In expectation, we have \(d^k\) such descendants, while about \((d\lambda)^k=(d^k)^{1/\zeta}\) are connected to the root by copy edges and hence carry its label. Call these vertices good. After applying the boundary noise, at least around \(a_q(d\lambda)^k\) of these copies remain correctly observed. For \(a_q(d\lambda)^k\) to be at least of order one--since otherwise the signal sent by good vertices whose labels are correctly observed vanishes--we require roughly \(d^k\asymp a_q^{-\zeta}\) observations. Among these observations, any fixed incorrect label can be erroneously produced by the boundary noise with expected number at most of order \(b_qd^k \asymp b_q a_q^{-\zeta}\), suggesting the basic separation \(b_q\ll a_q^\zeta\). The condition \(a_q(d\lambda)^{k_q}\to\infty\) then ensures that the tree is deep enough to sufficiently amplify the signals in the boundary. 

\begin{assumption}\label{assump:noise-matrix}
Let \(\Phi\) be a PGF of an \(\nn_0\)-valued random variable with finite mean \(\Phi'(1) = d\). Let \(\lambda\in(0,1)\) and \(\zeta := \frac{\log d}{\log (d\lambda)}\). Consider the sequences \((k_q)_{q\in\nn}\) of depth and \((\Delta^{(q)})_{q\in\nn}\) of noise matrices. We say that \(\Phi, \lambda, (k_q)\), and \((\Delta^{(q)})\) satisfy Assumption~\ref{assump:noise-matrix} if \begin{enumerate}
    \item \(d\lambda > 1\),
    \item \(a_q(d\lambda)^{k_q} \to \infty\), and
    \item there exists some integer \(s > \zeta\) such that \(\ee D^s < \infty\) and \(\frac{b_q^{2s-1}}{a_q^{2s\zeta}} \to 0\)
\end{enumerate} as \(q\to\infty\).
\end{assumption}

\begin{assumption}\label{assump:noise-matrix-2}
Let \(\Phi\) be a PGF of an \(\nn_0\)-valued random variable with finite mean \(\Phi'(1) = d\). Let \(\lambda\in(0,1)\) and \(\zeta := \frac{\log d}{\log (d\lambda)}\). Consider the sequences \((k_q)_{q\in\nn}\) of depth and \((\Delta^{(q)})_{q\in\nn}\) of noise matrices. We say that \(\Phi, \lambda, (k_q)\), and \((\Delta^{(q)})\) satisfy Assumption~\ref{assump:noise-matrix-2} if \begin{enumerate}
    \item \(d\lambda > 1\),
    \item \(\Phi\) has finite moments of every order,
    \item \(a_q(d\lambda)^{k_q} \to \infty\), and
    \item there exists some \(\eps > 0\) such that \(b_q = o(a_q^{\zeta +\eps})\)
\end{enumerate} as \(q\to\infty\).
\end{assumption}

Consider the following broadcast process with label space \([0,1]\) on a Galton-Watson tree. The root receives a uniform label in \([0,1]\), and independently along each edge, the child copies its parent's label with probability \(\lambda\), and otherwise receives a fresh \(\mathrm{Unif}[0,1]\) label. The model is almost surely ``collision-free," meaning with probability one, all mutated labels will be distinct from existing true labels. Let \[E_k^\star := \sup_{\hat\tau} \Pr\left( \hat\tau(T_{\le k},\tau_{L_k}^\star)=\tau_\rho^\star \right)\] denote the Bayes-optimal reconstruction probability at depth \(k\), where \(\tau^\star\in [0,1]^{|T|}\) denotes the labels in this continuous model. Since \(E_k^\star\) is nonincreasing in \(k\) and bounded, define \[E_\infty^\star(\Phi,\lambda) = E_\infty^\star := \lim_{k\to\infty} E_k^\star.\] We study this collision-free continuum model in detail in Section~\ref{sec:bayes-random-tree}.

\begin{thm}\label{thm:tree-theorem}
If \(\Phi, \lambda, (k_q)\), and \((\Delta^{(q)})\) satisfy Assumption~\ref{assump:noise-matrix}, then \[\ee\|X_{q, k_q} - W_{q, k_q}\| \to 0\] as \(q\to\infty\), where \(\|\cdot\|\) denotes the total variation norm, i.e., \(\|\mu - \nu \| = \sup_A|\mu(A)-\nu(A)|\). In particular, \[p_T(q), \tilde p_T(q, \Delta)\to E_\infty^\star\] as \(q\to\infty\).
\end{thm}

\begin{rem}
Although \(E_\infty^\star(\Phi,\lambda)\) is defined through the infinite-depth continuum model, Section~\ref{sec:bayes-random-tree} gives explicit nontrivial lower and upper bounds.
\end{rem}

Since any \(\Phi, \lambda, (k_q)\), and \((\Delta^{(q)})\) satisfying Assumption~\ref{assump:noise-matrix-2} satisfies Assumption~\ref{assump:noise-matrix}, we obtain the following corollary. Indeed, as \(s \to \infty\), we have that \(\frac{2s-1}{2s}\to 1\), so choosing \(s\) large enough so that \(\frac{2s\zeta}{2s-1} < \zeta +\eps\), then \(b_q^{2s-1} = o(a_q^{(2s-1)(\zeta+\eps)}) = o(a_q^{2s\zeta})\), so Assumption~\ref{assump:noise-matrix} follows.

\begin{cor}\label{cor:tree-cor}
If \(\Phi, \lambda, (k_q)\), and \((\Delta^{(q)})\) satisfy Assumption~\ref{assump:noise-matrix-2}, then \[\ee\|X_{q, k_q} - W_{q, k_q}\|\to 0 \text{ and }p_T(q), \tilde p_T(q, \Delta)\to E_\infty^\star\] as \(q\to\infty\).
\end{cor}

\begin{rem}
Since Poisson random variables and bounded offspring distributions have finite moments of every order, Assumption~\ref{assump:noise-matrix-2} applies to both of these standard classes of Galton-Watson trees.
\end{rem}

A noteworthy feature of Theorem~\ref{thm:tree-theorem} and Corollary~\ref{cor:tree-cor} is that the label-retention probability, the probability of retaining the true labels after applying the noise, need not be bounded away from zero. In particular, \(a_q\) may vanish as \(q\to\infty\), and each individual noisy leaf label may become increasingly unlikely to equal its true label. Thus, sufficiently many descendants can collectively amplify a signal that becomes arbitrarily weak at the level of any single observed vertex, provided \(a_q(d\lambda)^{k_q}\to\infty\) and the off-diagonal noise is sufficiently spread out across incorrect labels. We may consider the following example.

\begin{example}
Suppose \(\Phi\) has finite moments of every order and \(d\lambda > 1\). Take \(\Delta_{ii} = q^{-\theta}\) for all \(i\in[q]\) and \(\Delta_{ij} = \frac{1-q^{-\theta}}{q-1}\) for all \(i\neq j\), where \(\theta >0\). If \(\theta < \frac{1}{\zeta} = \frac{\log(d\lambda)}{\log d}\) and \(k_q -\frac{\theta}{\log(d\lambda)}\log q\to\infty,\) then by Corollary~\ref{cor:tree-cor}, the Bayes-optimal reconstruction accuracies are the same given the noiseless and noisy labels.
\end{example}

\subsection{Proof Idea}
We first assume that each true boundary label is retained with probability at least a fixed positive constant, and then reduce the general case to this setting.

We first study the continuum model in Section~\ref{sec:bayes-random-tree}, in which fresh labels are uniform on \([0,1]\) and distinct almost surely, while the copying rule remains unchanged. We show in Lemma~\ref{lem:two-witness} that a label appearing at two boundary vertices in different child-subtrees of a vertex determines that vertex's label exactly. We call such a pair ``two witnesses." Similar arguments using the same label in two different branches appear in~\cite{MS04, KMS16}.

To approximate the full root posterior, we summarize the boundary information by labels identified near the root. Starting at the root, we stop whenever two witnesses identify a vertex's label, retaining that label and discarding its descendant subtree. We truncate unresolved branches at depth \(m\), assign the uniform prior there, and combine these distributions upward using Bayes' rule.

Intuitively, for large \(m\), the labels identified before depth \(m\) should capture almost all the boundary information relevant to the root posterior. Suppose a vertex \(u\) at depth less than \(m\) is connected to the root by copy edges. If two paths of copy edges lead from \(u\) to the boundary through different children of \(u\), their boundary labels provide two witnesses identifying \(u\)’s label. Since \(u\) carries the root’s label, pruning retains that value. Thus, for the root label to appear at the boundary yet be discarded, all copy-edge paths from the root to the boundary must share their first \(m\) edges. The event that at least one such path exists and all such paths share these edges has probability at most \(\alpha^m\), for some constant \(\alpha\in(0,1)\). We show that its probability bounds the expected total variation error of the approximation.

The same approximation can be constructed from noisy observations. Conditional on the true labels, replacement draws are independent and atomless, so they create no false matches. If a vertex has copy-edge paths to arbitrarily large depths, it has increasingly many boundary copies of its label, making it unlikely that noise removes them all. Thus, for fixed \(m\), the noisy and noiseless pruning constructions agree with probability tending to one as the boundary depth grows. Together with the pruning estimate, then letting \(m\to\infty\) gives convergence of the full posteriors in expected total variation.

For finite \(q\), independent fresh redraws and erroneous boundary observations can create label equalities that do not reflect common copy ancestry. We call these accidental label collisions. Section~\ref{subsec:easy-noise-matrix} first treats the constant-retention regime, in which \(a_q \geq a_0 > 0\) and \(b_q \to 0\). On a tree with at most \(N_q\) vertices, a union bound over pairs gives a collision probability of at most \(O(N_q^2(q^{-1}+b_q))\). We therefore choose a neighborhood whose depth tends to infinity sufficiently slowly that it contains at most \(N_q\) vertices with high probability, where \(N_q^2(q^{-1}+b_q)\to0\). The neighborhood is then large enough for the continuum pruning approximation to become accurate, yet small enough that accidental collisions remain unlikely. A separate argument extends the result to arbitrary depths \(k_q\to\infty\) by estimating the labels on the boundary of this slowly growing neighborhood from the original noisy boundary observations, with errors satisfying the same constant-retention assumptions.

Finally, Section~\ref{subsec:bootstrap} compensates for vanishing retention by using more inherited copies. At distance \(h\), a vertex has \((d\lambda)^h\) descendants connected to it by copy edges in expectation. Each is correctly observed with probability at least \(a_q\), giving the signal scale \(a_q(d\lambda)^h\). We choose \(h_q\) so that this scale stays of constant order with a sufficiently large lower bound.

Greater depth also permits more misleading repetitions. We therefore require several occurrences of a candidate label in each of two distinct child-subtrees. One fresh redraw below the vertex can produce many copies within one child-subtree, but cannot account for occurrences in both. Moment estimates for the sizes of these families control misleading repetitions. Under Assumption~\ref{assump:noise-matrix}, this yields estimates at generation \(k_q-h_q\) whose correct-output probabilities are bounded away from zero and whose probabilities of any specified incorrect output vanish. Since \(a_q(d\lambda)^{k_q}\to\infty\) ensures \(k_q-h_q\to\infty\), a diverging depth remains on which to apply the constant-retention result of Section~\ref{subsec:easy-noise-matrix}.

\subsection{Acknowledgements}
The author thanks Youngtak Sohn for introducing the problem and for the helpful discussions on approaches.

\subsection{Statement of AI Use}
The author used OpenAI's ChatGPT to assist in checking and refining mathematical arguments, and with exposition and manuscript editing. The mathematical ideas, contributions, and final arguments are the author's, and the author takes full responsibility for the results and contents of the manuscript.

\section{Robust Reconstruction in the Continuum Model}\label{sec:bayes-random-tree}
Let \(T\) denote a Galton-Watson tree rooted at \(\rho\) with offspring PGF \(\Phi\) with a finite mean \(d = \Phi'(1)\). 

We study the continuum model introduced in Section~\ref{subsec:main-results}, in which distinct copy components receive distinct labels almost surely. This allows us to analyze reconstruction through repeated-label events. Section~\ref{subsec:easy-noise-matrix} will transfer the resulting estimates to the finite-alphabet model.

Let \(\mu\) be the Lebesgue measure on \([0,1]\). We generate the continuous broadcast process \((\tau_u^\star)_{u\in T}\) by \(\tau_\rho^\star\sim \mu\) and \[\tau_v^\star = \begin{cases}
    \tau_u^\star, &\text{with probability }\lambda,\\
    Y_v, Y_v\sim\mu, &\text{with probability }1-\lambda,
\end{cases}\] independently along every edge \(u\to v\). Since \(\mu\) is atomless, independently generated fresh labels are distinct almost surely. Consequently, if we declare an edge to be open whenever the child copies the parent, then almost surely, each connected component of the resulting open-edge forest receives a distinct label. On this a.s. event, \(\tau_u^\star = \tau_v^\star\) if and only if \(u\) and \(v\) belong to the same open component. This collision-free representation will allow us to characterize reconstruction through repeated-label events, specifically the two-witness event.

For the noisy observations, we introduce an auxiliary continuum channel having the same collision-free feature. The true label is retained with probability at least \(a_0 > 0\), while a noisy observation is replaced by an atomless random label. Importantly, the replacement law is allowed to depend on the true label. 

Let \(\tilde\tau^\star\) denote the noisy labels generated in accordance with some transition kernel \(\kappa\). In other words, for any Borel \(S\subset [0,1]\), \[\Pr(\tilde\tau_u^\star\in S | \tau_u^\star = x) = \kappa(x, S).\] The noise is generated independently for each vertex.

\begin{assumption}\label{assump:noise-kernel}
Let \(\Phi\) be a PGF of an \(\nn_0\)-valued random variable with finite mean \(\Phi'(1) = d\). Let \(\lambda\in(0,1)\). Let \(\kappa\) be some transition kernel. We say that \(\Phi, \lambda\), and \(\kappa\) satisfy Assumption~\ref{assump:noise-kernel} if \begin{enumerate}
    \item \(d\lambda > 1\),
    \item \(\Phi\) has finite second moment, and
    \item \[\kappa(x, S) = \Pr(\tilde\tau_u^\star\in S | \tau_u^\star = x) = a_0(x)\mathbb1\{x\in S\} + (1-a_0(x))\cdot\nu_x (S)\] for every Borel \(S\subset [0,1]\), where \(a_0(x)\geq a_0\) for all \(x\in[0,1]\), for some \(a_0\in(0,1]\), and \(x\mapsto \nu_x\) is a measurable Markov kernel on \([0,1]\) such that \(\nu_x\) is atomless for every \(x\in[0,1]\).
\end{enumerate}
\end{assumption}

Let $\hat{\tau}_{\rho}^\star(k)$ denote the maximum likelihood estimator of the root label $\tau_{\rho}^\star$ under the $[0, 1]$-valued broadcast process, given the observed labels at depth $k$. Similarly, let $\hat{\tilde{\tau}}_{\rho}^\star(k)$ denote the maximum likelihood estimator given the noisy observations in this continuous setting. 

For the continuum model, let \[E_k^\star := \Pr(\hat\tau_\rho^\star(k) = \tau_\rho^\star) = \sup_{\hat\tau}\Pr(\hat\tau(T, \tau_{L_k(\rho)}^\star) = \tau_\rho^\star)\] denote the Bayes-optimal probability from the true depth-\(k\) labels. Similarly, let \[\tilde E_k^\star := \Pr(\hat{\tilde\tau}_\rho^\star(k) = \tau_\rho^\star) = \sup_{\hat\tau}\Pr(\hat\tau(T, \tilde\tau_{L_k(\rho)}^\star) = \tau_\rho^\star)\] denote the Bayes-optimal probability from the noisy depth-\(k\) labels under Assumption~\ref{assump:noise-kernel}.

In this section, we will demonstrate that the optimal probabilities \(E_k^\star, \tilde E_k^\star\) coincide as \(k\to\infty\). Since \(E_k^\star\) is bounded and monotonically decreasing, the limit exists. As a result, we define \[E_\infty^\star := \lim_{k\to\infty} E_k^\star.\]

\subsection{Survival Probability}
\begin{defin}\label{defin:Z_rhok}
We define variables for the number of depth-\(k\) descendants carrying \(u\)'s label. \begin{align*}
    Z_{u, k} &= \sum_{v\in L_k(u)} \mathbb1(\tau_v^\star=\tau_u^\star),\\
    \tilde Z_{u, k} &= \sum_{v\in L_k(u)} \mathbb1(\tilde\tau_v^\star=\tau_u^\star).
\end{align*}
\end{defin}

We first study \(Z_{\rho, k}\). Recall that an edge is open whenever the child copies its parent's label. Since \(\mu\) is atomless, almost surely, a descendant has the same label as the root if and only if it is connected to the root by a path consisting entirely of open edges. We show that \((Z_{\rho, k})_{k\geq 0}\) is itself a Galton-Watson process.

Indeed, let \(D\sim\Phi\) denote the number of children of a vertex, and let \(R\) denote the number of its open children. Conditional on \(D\), \(R|D\sim\mathrm{Bin}(D, \lambda)\). Hence, the PGF of \(R\) is \[f(x) := \ee[x^R] = \Phi(1-\lambda + \lambda x)\] for any \(x\in[0,1]\). Its mean is \(f'(1) = \lambda\Phi'(1) = d\lambda > 1\), so \((Z_{\rho, k})_{k\geq 0}\) is a supercritical Galton-Watson process with offspring PGF \(f\).

Let \(\xi\in (0,1)\) denote its extinction probability, i.e., the probability that no true boundary label contains the true root label as the depth grows asymptotically. By Theorem I.5.1 of~\cite{AN72}, the extinction probability is the smallest solution of \[f(x)=x\] in \([0,1]\). Since \(f'(1)=d\lambda>1\), we have \(\xi<1\). In fact, \(f\) is strictly convex, so \(\xi\) is the unique solution to \(f(x)=x\) in \([0,1)\). Finally, \[\Pr(Z_{\rho,k}=0)=f^{\circ k}(0)\uparrow \xi.\]

Let \(\alpha:= f'(\xi) = \lambda\Phi'(1-\lambda + \lambda\xi)\). We observe that \(\alpha\in(0,1)\). Indeed, \(f\) is strictly increasing and strictly convex on \([0,1]\). Since \(f(\xi) = \xi, f(1) = 1\), and \(\xi < 1\), strict convexity implies \(f'(\xi) < \frac{f(1)-f(\xi)}{1-\xi} = 1\). Moreover, \(f'(\xi) > 0\).

We will also use the standard fact that the iterates of a noncritical Galton-Watson PGF converge geometrically to the extinction probability (see, e.g., Chapter I Section 11 of~\cite{AN72}). In particular, for every fixed \(x\in[0,1)\), there exist constants \(C_x > 0\) and \(\beta_x\in(0,1)\) such that \[|f^{\circ k}(x) - \xi| \leq C_x\beta_x^k\] for every \(k\in\nn\).

Define \(P_k := \Pr(Z_{\rho, k} = 0)\) and \(\tilde P_k(x) := \Pr(\tilde Z_{\rho, k} = 0 | \tau_\rho^\star=x)\) for \(x\in[0,1]\). In the noisy case, we condition on the true root label \(x\), since the transition kernel depends on \(\nu_x\), which need not be the same for each \(x\in[0,1]\).

\begin{lem}\label{lem:tildePk(x)=fk(1-a0x)}
For any \(x\in[0,1], k\in\nn_0\), and \(z\in\nn_0\), \[\tilde Z_{\rho,k} | (Z_{\rho, k}=z, \tau_\rho^\star = x)\sim\mathrm{Bin}(z, a_0(x)).\] Consequently, \[\tilde P_k(x) = f^{\circ k}(1-a_0(x)).\]
\end{lem}
\begin{proof}
On the event \(\tau_\rho^\star = x\), every vertex in the open component of the root has true label \(x\). Such a vertex is observed as \(x\) exactly when its true label is retained, which occurs independently with probability \(a_0(x)\).

Almost surely, a vertex outside the open component of the root cannot produce an additional observed label equal to \(x\). Indeed, if its true label is \(y\neq x\), retaining the label gives \(y\), while a replacement draw equals \(x\) with probability \(\nu_y(\{x\})=0\) by atomlessness. Since the tree up to depth \(k\) is finite almost surely, these measure-zero events may be discarded simultaneously. Hence, \[\tilde Z_{\rho, k} | (Z_{\rho, k} = z, \tau_\rho^\star = x)\sim\mathrm{Bin}(z, a_0(x)).\] 

It follows that \(\tilde P_k(x) = \ee[(1-a_0(x))^{Z_{\rho,k}}] = f^{\circ k}(1-a_0(x))\), where the second equality follows from the standard PGF recursion for a Galton-Watson process.
\end{proof}

\begin{lem}\label{lem:gw-geometric-convergence-to-xi}
Let \(d\lambda > 1\). There exist constants \(C= C(\Phi,\lambda, a_0)> 0\) and \(\beta = \beta(\Phi,\lambda,a_0) \in (0,1)\) such that, for every \(k\in\nn\), \[|P_k-\xi|, \sup_{x\in[0,1]}|\tilde P_k(x)-\xi| \leq C\beta^k.\]
\end{lem}
\begin{proof}
Since \(P_k = f^{\circ k}(0)\), the first inequality follows directly from the geometric convergence of the iterates of \(f\).

For the noisy case, Lemma~\ref{lem:tildePk(x)=fk(1-a0x)} gives \(\tilde P_k(x) = f^{\circ k}(1-a_0(x))\). Since \(a_0(x) \geq a_0\) and \(f\) is increasing, \[f^{\circ k}(0) \leq f^{\circ k}(1-a_0(x)) \leq f^{\circ k}(1-a_0).\] Both endpoints converge geometrically to \(\xi\), yielding our desired bound.
\end{proof}


\subsection{Bayes Reconstruction and the Pruning Approximation}\label{subsec:continuum-model}
In this subsection, we analyze the continuum broadcast model and develop a pruning approximation to the Bayes posterior. The quantitative estimates obtained here, particularly Lemmas~\ref{lem:E_k-S_km}-~\ref{lem:S_km-tilde S_km}, will be used in the finite-alphabet coupling argument of Section~\ref{subsec:easy-noise-matrix}. As an additional consequence, as we show in Proposition~\ref{prop:sec-tree-E_k=tilde-E_k}, they imply that the noisy and noiseless Bayes posteriors themselves become asymptotically equivalent in the continuum model.

Throughout this subsection, Assumption~\ref{assump:noise-kernel} holds, and we work on the probability-one event that distinct open components have distinct labels.

Conditionally independent given \((\tau_u^\star)_{u\in T}\), assume the noisy labels satisfy Assumption~\ref{assump:noise-kernel}, i.e., \[\tilde\tau_u^\star =\begin{cases}
    \tau_u^\star &\text{w.p. }a_0(\tau_u^\star),\\
    \tilde Y_u\sim \nu_{\tau_u^\star} &\text{w.p. }1 - a_0(\tau_u^\star),
\end{cases}\] where \(a_0(\tau_u^\star) \geq a_0\) and \(\nu_{\tau_u^\star}\) is an atomless distribution.

We observe that an a.s. sufficient condition for reconstructing the label is to observe at least two depth-\(k\) ``witness" vertices in two separate child-subtrees.

\begin{defin}(Least Common Ancestor)
Given two vertices \(u, v\in T\), the \emph{least common ancestor} (\emph{LCA}) is the common ancestor of greatest depth of \(u\) and \(v\). Denote the \emph{LCA} by \[\operatorname{LCA}(u,v),\] and the height of the LCA by \[h(u,v).\]
\end{defin}

\begin{defin}
Let \(\sigma\) be an arbitrary labeling on \(T\). Then, for any \(u\in T\), define \[J_{u,k}(\sigma) := \{j\in[0,1]: \exists v\neq w\in L_k(u)\text{ s.t. }\sigma_v=\sigma_w=j\text{ and }\operatorname{LCA}(v,w) = u\}.\] Equivalently, $j \in J_{u,k}$ if and only if the label $j$ appears in at least two descendants at distance $k$ from $u$ belonging to distinct child-subtrees of $u$.
\end{defin}

\begin{lem}\label{lem:two-witness}
Suppose \(\sigma\) is either \(\tau^\star\) or \(\tilde\tau^\star\), and fix $u \in T$. Then:
\begin{enumerate}
    \item Almost surely, $J_{u,k}(\sigma) \subseteq \{\tau_u^\star\}$. In particular, $|J_{u,k}(\sigma)| \leq 1$ a.s.
    \item If $J_{u,k}(\sigma) = \{j\}$, then $\tau_u^\star = j$ almost surely. 
\end{enumerate}
\end{lem}

\begin{proof}
We first show the case when \(\sigma=\tau^\star\). Let \(J_{u,k} = J_{u,k}(\tau^\star)\).

Let $v\neq w\in L_k(u)$ lie in distinct child-subtrees of $u$ and satisfy $\tau_v^\star=\tau_w^\star=j$. The $u\!\to\!v$ and $u\!\to\!w$ paths are edge-disjoint. 

If $j\neq \tau_u^\star$, then along each path there must be a first mutation whose new value equals the same fixed $j$; two independent $\operatorname{Unif}[0,1]$ draws coincide at $j$ with probability $0$. If $j=\tau_u^\star$ but equality to $\tau_u^\star$ is ever broken and later regained along a path, some mutation must draw exactly $\tau_u^\star$, again with probability $0$. Hence, the only positive-probability mechanism is pure copying of $\tau_u^\star$ along both paths, i.e., $j=\tau_u^\star$. This proves $J_{u,k}\subseteq\{\tau_u^\star\}$ a.s., and if $J_{u,k}=\{j\}$ then $\tau_u^\star=j$ a.s.

For the case when \(\sigma=\tilde\tau^\star\), condition on the true labels \(\tau^\star\). Under Assumption~\ref{assump:noise-kernel}, the noises at different vertices are conditionally independent, and their continuous components have no atoms. Thus, for any distinct \(v,w\) with \(\tau_v^\star \neq \tau_w^\star\), we have \(\Pr(\tilde\tau_v^\star = \tilde\tau_w^\star \mid \tau^\star) = 0\). Indeed, a replacement draw cannot equal a prescribed retained label with positive probability, and two independent replacement draws from \(\tau_w^\star\) and \(\tau_v^\star\) coincide with probability \[\int \nu_{\tau_w^\star}(\{z\})\nu_{\tau_v^\star}(dz) =0.\] Since the tree up to depth-\(k\) is finite a.s., it follows that almost surely every equality \(\tilde\tau_v^\star = \tilde\tau_w^\star\) implies \(\tau_v^\star = \tau_w^\star\). Combining this with the established \(\tau^\star\)-case yields \(J_{u,k}(\tilde\tau^\star) \subseteq \{\tau_u^\star\}\) a.s., and in particular \(J_{u,k}(\tilde\tau^\star) = \{j\}\) implies \(\tau_u^\star = j\) a.s.
\end{proof}

Let \(A_k:=\{J_{\rho,k}(\tau^\star)\neq\emptyset\}\) and \(\tilde A_k:=\{J_{\rho,k}(\tilde\tau^\star)\neq\emptyset\}.\) Define \[\omega(x) := 1-f(x) - (1-x)f'(x)\] for \(x\in[0,1)\).

\begin{lem}\label{lem:lim-of-two-witness}
If $d\lambda>1$, then \[\Pr(A_k) = \omega(P_{k-1})\text{ and }\Pr(\tilde A_k | \tau_\rho^\star = x) = \omega(\tilde P_{k-1}(x)).\] Moreover, there exists \(C > 0\) and \(\gamma \in(0,1)\) such that \[\left|\Pr(A_k) - (1-\xi)(1-\alpha)\right|, \sup_{x\in[0,1]} \left|\Pr(\tilde A_k|\tau_\rho^\star = x) - (1-\xi)(1-\alpha)\right| \leq C\gamma^k.\] Consequently, \(\Pr(A_k)\) and \(\Pr(\tilde A_k|\tau_\rho^\star=x)\) converge to \((1-\xi)(1-\alpha)\).
\end{lem}
\begin{proof}
Fix the root label \(x\). Let $R$ be the number of open children of the root. Then $R$ has probability-generating function \(\ee[t^R]=f(t).\)

For each open child $u$, let \[B_u:=\mathbb 1\{Z_{u,k-1}\ge 1\} \text{ and }\tilde B_u:=\mathbb 1\{\tilde Z_{u,k-1}\ge 1\}.\] Conditioned on $R=r$, the relevant child-subtrees are independent, and \[B_u \stackrel{\mathrm{iid}}{\sim} \mathrm{Bernoulli}(1-P_{k-1}) \text{ and }\tilde B_u \stackrel{\mathrm{iid}}{\sim} \mathrm{Bernoulli}(1-\tilde P_{k-1}(x)).\] By Lemma~\ref{lem:two-witness}, \(A_k\) occurs exactly when at least two of the \(B_u\) equal one, and \(\tilde A_k\) occurs exactly when at least two of the \(\tilde B_u\) equal one. Therefore, \[\Pr(A_k) = 1-f(P_{k-1}) - (1-P_{k-1})f'(P_{k-1}) = \omega(P_{k-1}).\] The same calculation conditional on \(\tau_\rho^\star=x\) gives \(\Pr(\tilde A_k | \tau_\rho^\star = x)=\omega(\tilde P_{k-1}(x))\).

Since \(f(\xi) = \xi\) and \(f'(\xi)=\alpha\), \(\omega(\xi) = (1-\xi)(1-\alpha).\) The convergence follows from \(P_k, \tilde P_k(x)\to \xi\) by Lemma~\ref{lem:gw-geometric-convergence-to-xi}. For the bound, note that all relevant arguments of \(\omega\) lie in a compact interval \([0,M]\subset [0,1)\) with \(M = \max\{1-a_0, \xi\}\). The PGF is smooth on this compact interval and \(\omega'(t) = -(1-t)f''(t)\), so \(\omega\) is Lipschitz on the interval. Combining this with the bounds for \(P_k\) and \(\tilde P_k(x)\) yields our desired inequality.
\end{proof}

Two witnesses determine the root label with certainty, but the Bayes estimator can also succeed when no such pair exists at the root. For example, suppose that the label of a child \(v\) of the root can itself be identified from two witnesses below \(v\). Then \(\tau_v^\star\) is known exactly, and since the edge \(\rho\to v\) copies with probability \(\lambda\), the label \(\tau_v^\star\) forms a positive-mass atom in the posterior distribution of \(\tau_\rho^\star\). Since we know by Lemma~\ref{lem:lim-of-two-witness} that the probability of two witnesses converges to \((1-\xi)(1-\alpha)\) and that the probability of the true label surviving is \(1-\xi\), which is a necessary condition for reconstruction, we know \(0 < (1-\xi)(1-\alpha) \leq E_\infty^\star \leq 1-\xi\). 

To study Bayes' estimator, we wish to recursively study the posterior measure for the label of each vertex. For any vertex \(u\in T\) at depth at most \(k\), consider the posterior measure \[\pi_{u,k}^\star(B) = \Pr(\tau_u^\star\in B | T, \tau_{L_{k-h(u)}(u)}^\star)\] where \(B\) is Borel and \(h(u) = d(\rho, u)\) is the height of \(u\). We denote the Dirac measure at \(a\) by \(\delta_a\).

\begin{lem}\label{lem:posterior-measure-recursion}
Fix a finite subtree rooted at \(u\), and call its leaves the terminal vertices. Suppose the true labels of the terminal vertices are observed, and let \(S_u\) denote the set of labels appearing among them. Then, \[\pi_{u,k}^\star = c_u\mu + \sum_{a\in S_u} p_u(a)\delta_a,\] where \(c_u + \sum_{a\in S_u} p_{u,k}(a) = 1\). Furthermore, the following hold almost surely. \begin{enumerate}
    \item Every atom of \(\pi_{v,k}^\star\) is one of the terminal labels observed below \(v\). \label{statement-1}
    \item If \(S_v\) and \(S_w\) contain the same atom \(a\) for distinct children vertices \(v, w\in L_1(u)\), then \(\pi_{u,k}^\star = \delta_a\). \label{statement-2}
    \item If the sets \((S_v)_{v\in L_1(u)}\) are pairwise disjoint, then \(c_u > 0\) and, for \(a \in S_v\), \[p_{u,k}(a) = \frac{\lambda c_up_{v,k}(a)}{1-\lambda + \lambda c_v}.\]\label{statement-3}
\end{enumerate}
\end{lem}
\begin{proof}
Since we condition on the event that distinct open components have distinct labels, every observed label in \(S_u\) identifies a unique open component.

Firstly, statement~\eqref{statement-2} holds by Lemma~\ref{lem:two-witness}. We prove the structure of the posterior measure and statements~\eqref{statement-1} and~\eqref{statement-3} by induction upward from the terminal vertices. Clearly, for any terminal vertex \(v\), the posterior measure is \(\pi_{v,k}^\star = \delta_{\tau_v^\star}\). 

Now, assume inductively that for every \(v\in L_1(u)\), \(\pi_{v,k}^\star = c_v\mu + \sum_{a\in S_v} p_v(a)\delta_a\). We also assume that the sets \((S_v)_{v\in L_1(u)}\) are pairwise disjoint, since otherwise, by statement~\eqref{statement-2}, we have that \(\pi_{u,k}^\star = \delta_a\) for some \(a\in S_u\).

We know that the joint law of \((\tau_u^\star, \tau_v^\star)\) for \(v\) child of \(u\) is \[\lambda\mu(dx)\delta_x(dy) + (1-\lambda)\mu(dx)\mu(dy).\] Since the joint law is symmetric in \(x\) and \(y\), then \[\mathrm{Law}(\tau_u^\star|\tau_v^\star = a) = (1-\lambda)\mu + \lambda \delta_a.\] Then, for every Borel set \(B\), \[\Pr(\tau_u^\star\in B | \mathcal F_v) = \ee[\Pr(\tau_u^\star\in B|\tau_v^\star) | \mathcal F_v] = (1-\lambda)\mu(B) + \lambda\pi_{v,k}^\star(B),\] where \(\mathcal F_v = \sigma(\tau_w^\star | w\in L_{k-h(v)}(v))\). Hence, \[\mathrm{Law}(\tau_u^\star | \mathcal F_v) = (1-\lambda + \lambda c_v)\mu + \sum_{a \in S_v} \lambda p_v(a) \delta_a.\] 

We now combine the child observations. For a fixed child \(v\), up to a common normalizing factor depending only on the observations below \(v\), the coefficient \(c_v\) is the weight of configurations in which the open component of \(v\) does not meet an observed terminal component, while \(p_v(a)\) is the weight of configurations in which it meets the observed component labeled \(a\). Across the edge \(u\to v\), a closed edge contributes weight \(1-\lambda\), while an open edge contributes \(\lambda c_v\) in the former case and \(\lambda p_v(a)\) in the latter. Hence, child \(v\) contributes background weight \(1-\lambda + \lambda c_v\) and atomic weight \(\lambda p_v(a)\) at \(a\in S_v\). The common normalizing factor is independent of the value of \(\tau_u^\star\), meaning it cancels when the posterior at \(u\) is normalized.

Since the child-subtrees are conditionally independent given \(\tau_u^\star\), these weights multiply across children. Moreover, since the atom sets are pairwise disjoint, it must be the case that the open component of \(u\) either meets no observed terminal component or it meets the observed component labeled \(a\), for a unique \(a \in S_v\) and a unique child \(v\). Thus, the unnormalized posterior measure at \(u\) is \[\pi_{u,k}' = \left(\prod_{v\in L_1(u)} (1-\lambda + \lambda c_v)\right)\mu + \sum_{v\in L_1(u)}\sum_{a\in S_v} \lambda p_v(a)\left(\prod_{\substack{w\in L_1(u)\\ w\neq v}} (1-\lambda + \lambda c_w)\right)\delta_a.\] Indeed, in the first case, every child branch is in its nonatomic state, producing the first product. In the second case, child \(v\) identifies the component with \(a\), while every other child must remain in its nonatomic state. Letting \(Z_u\) be the normalizing constant, we conclude that \(\pi_{u,k}^\star = c_u\mu + \sum_{a\in S_u} p_u(a)\delta_a\) for \(c_u = \frac{\prod_v (1-\lambda +\lambda c_v)}{Z_u}\) and \(p_u(a) = \frac{\lambda p_v(a) \prod_{w\neq v}(1-\lambda + \lambda c_w)}{Z_u} = \frac{\lambda c_u p_v(a)}{1-\lambda + \lambda c_v}\). Since each \(1-\lambda +\lambda c_v > 0\) in the first product, we know that \(c_u > 0\).
\end{proof} 

With Lemma~\ref{lem:two-witness} and the posterior measures in hand, we define a pruning operator that collapses a depth–$k$ tree at the first internal vertex exhibiting a ``two-witness" signal up to the \(m\)th generation vertices. 

\begin{defin}\label{def:T_*}
Let $T\subset\mathbb N^*$ be a rooted tree with root $\rho$, label \(\sigma\), and depth \(k\). Let \(m \leq k\). Define the pruned subtree \[T_{k, m} = T_{k, m}(T,\sigma_{L_k(\rho)})\] recursively, processing vertices in increasing depth $i=0,1,\dots,m-1$ as follows. First, remove all leaf vertices without a label, i.e., all leaf vertices with depth less than \(k\). Do this repeatedly until all leaf vertices have a label. Consider each \(u\in L_i(\rho)\). \begin{enumerate}
    \item If \(u\) has two witnesses with label \(a\), then remove the subtree rooted below \(u\) and record the label of \(u\) as a. We call \(u\) a witness stop.
    \item If \(u\) does not have two witnesses, continue to its children. If it does not have any children, leave it unlabeled.
    \item At generation \(m\), stop without recording a label, and remove the subtree below it.
\end{enumerate}
The resulting tree, $T_{k, m}$, is a subtree of $T$ of depth at most \(m\). We call a vertex active if it is reached by the pruning procedure and is neither a witness stop nor a generation-\(m\) truncation vertex. An active vertex with at least one retained child is called an active internal vertex.
\end{defin}

We now attach a synthetic measure \(\pi_{u,k,m}^\star\) on each vertex \(u\). We will later show that using Bayes' estimator on the synthetic posterior measure performs almost as well as Bayes' estimator on the true posterior measure. \begin{enumerate}
    \item At a witness stop \(u\) with label \(a\), let \(\pi_{u,k,m}^\star = \delta_a\).
    \item At an unresolved generation-\(m\) vertex or on a branch with no level-\(k\) descendant, set \(\pi_{u,k,m}^\star = \mu\).
    \item At an active internal vertex, combine the child synthetic measures by the recursion of Lemma~\ref{lem:posterior-measure-recursion}.
\end{enumerate}

Let \(\psi_{k, m}^\star\) output an atom of \(\pi_{\rho,k,m}^\star\) with maximal mass. Here, we may use an arbitrary tie-breaking strategy. If \(\pi_{\rho,k,m}^\star\) is atomless, output an arbitrary label. Define \[S_{k,m} = \Pr(\psi_{k,m}^\star = \tau_\rho^\star).\] 

For the noisy case, we define the true posterior measure using the noisy labels to be \(\tilde\pi_{u,k}^\star\). Similarly define \(\tilde T_{k,m} = T_{k,m}(T, \tilde\tau_{L_k(\rho)}^\star)\), \(\tilde\pi_{u,k,m}^\star\), \(\tilde\psi_{k,m}^\star\), and \(\tilde S_{k,m}\). We note that \(\pi_{u,k}^\star\) and \(\tilde\pi_{u,k}^\star\) are the true posterior of \(u\) using the noiseless and noisy labels, respectively, and \(\pi_{u,k,m}^\star\) and \(\tilde\pi_{u,k,m}^\star\) are the auxiliary synthetic measures with noiseless and noisy labels, respectively. 

The next three lemmas bound the pruning error and compare the noisy and noiseless synthetic measures.

Let \(U_{k,m}\) be the event that the open component of the root reaches level \(k\), but for each generation \(0,\dots,m-1\), the unique vertex of the root open component lying on the surviving branch has exactly one open child whose open component reaches level \(k\). Equivalently, the root label appears at level \(k\), but no witness stop exposing that label occurs in generations \(0,\dots, m-1\).

\begin{lem}\label{lem:E_k-S_km}
For any \(m\in [k]\), \[\ee\left\|\pi_{\rho,k}^\star - \pi_{\rho,k,m}^\star\right\| \leq \Pr(U_{k,m}) =: u_{k,m}.\] In particular, \[0\leq E_k^\star-S_{k,m} \leq u_{k,m}.\]
\end{lem}
\begin{proof}
We show that pruning can only increase the mass of each retained atom and of the nonatomic part, so the total variation error is exactly the original posterior mass of the discarded atoms. Let \(\mathcal F_k = \sigma(T, \tau_{L_k(\rho)}^\star)\). At every vertex \(u\) in \(T_{k,m}\), call an atom in the full posterior \(\pi_{u,k}^\star\) old if it is also an atom in \(\pi_{u,k,m}^\star\). Otherwise, call it new. Let \(O_u\subset S_u\) be the set of old atoms, where \(S_u\) is the set of all atoms in the full posterior. Then, \(O_u\) is the set of all atoms in \(\pi_{u,k,m}^\star\). Let \(N_u = S_u\setminus O_u\).

Throughout the proof, let \(\pi_u := \pi_{u, k}^\star\) and \(\bar\pi_u := \pi_{u,k,m}^\star\). If \(u\) is an active vertex of \(T_{k,m}\) that is not a witness stop, then the atom sets coming from distinct children are disjoint. Then \(c_u, \bar c_u > 0\), where \(\pi_{u} = c_u\mu + \sum_{a\in S_u}p_u(a)\delta_a\) and \(\bar\pi_{u} = \bar c_u\mu + \sum_{a\in O_u} \bar p_u(a) \delta_a\). For such a vertex, define \[r_u(a) = \frac{p_u(a)}{c_u} \text{ and } \bar r_u(b) = \frac{\bar p_u(b)}{\bar c_u},\] where \(a\in S_u\) and \(b\in O_u\). We first inductively show the following: \begin{itemize}
    \item \(\bar p_u(a) \geq p_u(a)\) and \(\bar c_u \geq c_u\) for all \(u\in T_{k,m}\) and \(a\in O_u\).
    \item \(\bar r_u(a) \geq r_u(a)\) for all active vertex \(u\) that is not a witness stop and \(a\in O_u\).
    \item \(\sum_{a\in O_u} (\bar p_u(a) - p_u(a)) \leq \sum_{b\in N_u} p_u(b)\).
\end{itemize}

Consider the base case: any leaf \(u\) of \(T_{k,m}\). We split into two cases: either \(u\) is a witness stop or not. In the former case, \(O_u = \{a\}\), implying \(\bar p_u(a) = 1 = p_u(a)\), \(\bar c_u = 0 = c_u\), and \(\bar r_u(a) \geq r_u(a)\) vacuously holds. Also, \(\sum_{a\in O_u} (\bar p_u(a) - p_u(a)) = 0 = \sum_{b\in N_u} p_u(b)\). In the latter case, \(O_u = \emptyset\), meaning \(\bar p_u(a) \geq p_u(a)\) vacuously holds, \(\bar c_u = 1 \geq c_u\), and \(\bar r_u(a) \geq r_u(a)\) vacuously holds. Moreover, \(\sum_{a\in O_u}(\bar p_u(a) - p_u(a)) = 0 \leq \sum_{b\in N_u} p_u(b)\).

Assume the inductive hypothesis, i.e., for any active internal vertex \(u\in T_{k,m}\) that is not a witness stop, the four desired inequalities hold for every \(v\in L_1(u) \cap T_{k,m}\). Since no atom occurs in two distinct child-subtrees, the sets \(O_u\) and \(N_u\) are the disjoint unions of the corresponding old and new atom sets contributed by the children. We first show that \[\underbrace{\sum_{a\in O_u} \bar r_u(a)}_A \leq \underbrace{\sum_{a\in O_u} r_u(a)}_B + \underbrace{\sum_{b\in N_u} r_u(b)}_C.\] Observe that, by Lemma~\ref{lem:posterior-measure-recursion}, \begin{align*}
    A &= \sum_{v\in L_1(u)\cap T_{k,m}}\sum_{a\in O_v} \frac{\lambda \bar p_v(a)}{1-\lambda + \lambda \bar c_v}\\
    &= \sum_{v\in L_1(u)\cap T_{k,m}} \frac{\lambda (1-\bar c_v)}{1-\lambda + \lambda \bar c_v}\\
    [c_v \leq \bar c_v]&\leq \sum_{v\in L_1(u)\cap T_{k,m}} \frac{\lambda (1-c_v)}{1-\lambda + \lambda c_v}\\
    &= \sum_{v\in L_1(u)\cap T_{k,m}} \left[\sum_{a\in O_v}\frac{\lambda p_v(a)}{1-\lambda + \lambda c_v} + \sum_{b\in N_v}\frac{\lambda p_v(b)}{1-\lambda + \lambda c_v}\right]\\
    &= B + C
\end{align*} as desired. We now observe that \(\sum_{a\in O_u} \bar p_u(a) = \frac{A}{1 + A}, \sum_{a\in O_u} p_u(a) = \frac{B}{1 + B+C}, \sum_{a\in N_u} p_u(a) = \frac{C}{1 + B+C}, \bar c_u = \frac{1}{1 + A},\) and \(c_u = \frac{1}{1 + B + C}.\) First, \(\bar c_u = \frac{1}{ 1 + A} \geq \frac{1}{1 + B+C} = c_u\) because \(A \leq B+C\). Also, \(\frac{A}{1 +A} \leq \frac{B+C}{1+B+C}\) because \(t\mapsto \frac{t}{1+t}\) is increasing. Hence, \[\sum_{a\in O_u} (\bar p_u(a) - p_u(a))  = \frac{A}{1+A} - \frac{B}{1+B+C} \leq \frac{B+C}{1+B+C}-\frac{B}{1+B+C} =\sum_{b\in N_u} p_u(b).\] Next, fix \(a\in O_u\), and let \(v\) be the unique child of \(u\) whose subtree contains \(a\). If \(v\) is a witness stop, then the full and synthetic posteriors at \(v\) are both \(\delta_a\), and therefore, \[\bar r_u(a) = \frac{\lambda}{1-\lambda} = r_u(a).\] Otherwise, because \(a\) is old, \(v\) is an active vertex rather than an unresolved truncated vertex. Hence, \(c_v, \bar c_v > 0\), and the inductive hypotheses give \begin{align*}
    \bar r_u(a) &= \frac{\lambda \bar p_v(a)}{1-\lambda + \lambda \bar c_v}\\
    &= \frac{\lambda \bar r_v(a)}{\frac{1-\lambda}{\bar c_v} + \lambda}\\
    [\bar r_v(a) \geq r_v(a)]&\geq \frac{\lambda r_v(a)}{\frac{1-\lambda}{\bar c_v} + \lambda}\\
    [\bar c_v \geq c_v]&\geq \frac{\lambda r_v(a)}{\frac{1-\lambda}{c_v} + \lambda}\\
    &= r_u(a).
\end{align*} Finally, \[\bar p_u(a) = \frac{\bar r_u(a)}{1+A} \geq \frac{r_u(a)}{1 + A} \geq \frac{r_u(a)}{1+B+C} = p_u(a),\] where the first inequality follows from \(\bar r_u(a) \geq r_u\) and the second inequality follows from \(A \leq B+C\). This completes the induction.

Conditional on \(\cF_k\), we know that \[\bar \pi_{\rho} - \pi_\rho = \underbrace{(\bar c_\rho-c_\rho)\mu + \sum_{a\in O_\rho} (\bar p_\rho(a) - p_\rho(a))\delta_a}_{\text{I}} - \underbrace{\sum_{b\in N_\rho}p_\rho(b)\delta_b}_{\text{II}}.\] From our induction, we know that the signed measure \(\bar\pi_\rho - \pi_\rho\) has its positive part supported on the nonatomic component and the old atoms, whereas its negative part is supported exactly on the new atoms. Since the two probability measures have the same total mass, the positive and negative parts have equal total mass. Thus, \[\left\|\pi_\rho-\bar\pi_\rho\right\| = \sum_{b\in N_\rho} p_\rho(b).\] Since \(N_\rho\) is \(\cF_k\)-measurable, \(\sum_{b\in N_\rho}p_\rho(b) = \Pr(\tau_\rho^\star\in N_\rho | \cF_k)\). Taking the expectation on both sides, we obtain \[\ee\left\|\pi_\rho - \bar\pi_\rho\right\| = \Pr(\{\tau_\rho^\star\text{ is a new atom}\}).\] We now note that \(\{\tau_\rho^\star\text{ is a new atom}\} \subset U_{k,m}\). Indeed, if the root label is a new atom, its open component reaches level \(k\), but no vertex on that component in generations \(0,\dots, m-1\) has two children continuing to level \(k\). Otherwise, the root label would have been recorded at a witness stop, making it an old atom. Thus, \[\ee\left\|\pi_\rho-\bar\pi_\rho\right\| \leq \Pr(U_{k,m}).\]

Conditional on \(\mathcal F_k\), let \(a^\star\) and \(\bar a\) be the output of Bayes' estimator and \(\psi_{k,m}^\star\), respectively. We use the convention \(p_\rho(\bar a) = 0\) if the synthetic posterior has no atoms. If the full posterior has no atoms, the desired conditional inequality is immediate, so suppose it has at least one atom. If \(a^\star\in N_\rho\), then \[p_\rho(a^\star) - p_\rho(\bar a) \leq \sum_{b\in N_\rho} p_\rho(b)\] because the sum contains \(p_\rho(a^\star)\). If \(a^\star\in O_\rho\), then \[p_\rho(a^\star) \leq \bar p_\rho(a^\star) \leq \bar p_\rho(\bar a),\] meaning \[p_\rho(a^\star) - p_\rho(\bar a)\leq \bar p_\rho(\bar a) - p_\rho(\bar a) \leq  \sum_{a\in O_\rho} (\bar p_\rho(a) - p_\rho(a)) \leq \sum_{b\in N_\rho} p_\rho(b).\] Hence, conditionally on \(\mathcal F_k\), \[\max_a p_\rho(a) - \Pr(\psi_{k,m}^\star = \tau_\rho^\star| \mathcal F_k) = p_\rho(a^\star) - p_\rho(\bar a) \leq \sum_{b\in N_\rho} p_\rho(b).\] Taking the expectation on both sides, \[0\leq E_k^\star - S_{k,m} \leq \Pr(\tau_\rho^\star\text{ is a new atom}) \leq \Pr(U_{k,m})\] as desired.
\end{proof}

We now compute the limit of \(u_{k,m} = \Pr(U_{k,m})\) for fixed \(m\).

\begin{lem}\label{lem:u_km}
For any fixed \(m \in \zz_{\geq0}\), \[u_{k,m} = (1-P_{k-m})\prod_{j=k-m}^{k-1}f'(P_j) \leq \alpha^m.\]
\end{lem}

\begin{proof}
We induct on \(m\). For \(m = 0\), \(u_{k,0} = 1 - P_k \leq \alpha^0\). For \(m \geq 1\), exactly one open child must realize \(U_{k-1, m-1}\), while every other open child must fail to reach generation \(k-1\). If \(R\) is the number of open children, whose PGF we know is \(f\), then \[u_{k,m} = \ee_R[Ru_{k-1,m-1} P_{k-1}^{R-1}] = f'(P_{k-1})u_{k-1,m-1} = (1-P_{k-m})\prod_{j=k-m}^{k-1}f'(P_j).\] Noting that \(P_j \leq \xi\), then \(f'(P_j) \leq f'(\xi)=\alpha\), concluding our proof.
\end{proof}

It remains to bound \(|S_{k,m} - \tilde S_{k,m}|\).
\begin{lem}\label{lem:S_km-tilde S_km}
For any \(k \geq m\), \[\ee\left\|\pi_{\rho,k,m}^\star-\tilde\pi_{\rho,k,m}^\star\right\| \leq C\sum_{r=0}^{m-1}d^r\gamma^{k-r}\] and \[\left|S_{k,m} - \tilde S_{k,m}\right| \leq C\sum_{r=0}^{m-1}d^r\gamma^{k-r}\] for some constants \(C > 0\) and \(\gamma\in(0,1)\) depending only on \(\Phi, \lambda\), and \(a_0\).
\end{lem}
\begin{proof}
Since the total variation distance between two probability measures is at most one, \[\ee\left\|\pi_{\rho,k,m}^\star-\tilde\pi_{\rho,k,m}^\star\right\| \leq \Pr\left(\pi_{\rho,k,m}^\star\neq\tilde\pi_{\rho,k,m}^\star\right).\]

If \(\pi_{\rho,k,m}^\star \neq \tilde \pi_{\rho,k,m}^\star\), then at some vertex \(u\in \bigcup_{r=0}^{m-1} L_r(\rho)\), the two constructions of \(T_{k,m}\) and \(\tilde T_{k,m}\) must make different decisions: either one stops at \(u\) while the other continues to its children, or both stop at \(u\) but record different labels.

Indeed, if the two constructions make the same decision, with the same recorded label whenever they stop, at every vertex in the first \(m\) generations, then the resulting pruned trees are identical with identical terminal messages. Since both synthetic measures are then combined upward using the same numerical recursion from Lemma~\ref{lem:posterior-measure-recursion}, we have \(\pi_{\rho,k,m}^\star = \tilde\pi_{\rho,k,m}^\star\).

Define \[M_{u, \ell} = \begin{cases}
    a &\text{if }J_{u,\ell}(\tau^\star) = \{a\},\\
    \dagger &\text{if }J_{u,\ell}(\tau^\star) = \emptyset,
\end{cases}\text{ and }\tilde M_{u, \ell} = \begin{cases}
    a &\text{if }J_{u,\ell}(\tilde\tau^\star) = \{a\},\\
    \dagger &\text{if }J_{u,\ell}(\tilde\tau^\star) = \emptyset.
\end{cases}\] Hence, \[\{\pi_{\rho, k,m}^\star\neq \tilde\pi_{\rho, k,m}^\star\} \subset \bigcup_{r=0}^{m-1} \bigcup_{u\in L_r(\rho)} \{M_{u,k-r} \neq \tilde M_{u,k-r}\} =: \cE_{k,m}.\]

We now observe that \[|S_{k,m} - \tilde S_{k,m}| = |\Pr(\psi_{k,m}^\star=\tau_\rho^\star, \tilde \psi_{k,m}^\star\neq\tau_\rho^\star) - \Pr(\psi_{k,m}^\star\neq\tau_\rho^\star, \tilde \psi_{k,m}^\star=\tau_\rho^\star)| \leq \Pr(\psi_{k,m}^\star\neq \tilde\psi_{k,m}^\star).\] 

The exact same argument as before yields \[\{\psi_{k,m}^\star\neq \tilde\psi_{k,m}^\star\} \subset \cE_{k,m}.\] Hence, it suffices to show that \[\Pr(\cE_{k,m}) \leq C\sum_{r=0}^{m-1}d^r \gamma^{k-r}.\]

Consider \[\eps_\ell =\sup_{x\in[0,1]}\Pr(M_{u,\ell}\neq \tilde M_{u,\ell} | \tau_u^\star = x).\] By Lemma~\ref{lem:two-witness}, \(\{\tilde M_{u,\ell}\neq \dagger\} \subset \{M_{u,\ell}\neq \dagger\}\) almost surely, and whenever \(\tilde M_{u,\ell} \neq \dagger\), then \(\tilde M_{u,\ell} = M_{u,\ell} =\tau_u^\star\) almost surely. Hence, conditional on \(\tau_u^\star = x\), the two marks disagree exactly when a noiseless witness exists, but a noisy witness does not. Then, \begin{align*}
    \Pr(M_{u,\ell} \neq \tilde M_{u,\ell}|\tau_u^\star = x) & = \Pr(A_\ell) - \Pr(\tilde A_\ell | \tau_\rho^\star = x)\\
    &\leq |\Pr(A_\ell) - (1-\xi)(1-\alpha)| + |\Pr(\tilde A_\ell |\tau_\rho^\star = x) - (1-\xi)(1-\alpha)|.
\end{align*} The uniform bound by Lemma~\ref{lem:lim-of-two-witness} then yields \[\eps_\ell \leq C\gamma^\ell\] for some \(C > 0\) and \(\gamma\in(0,1)\) depending only on \(\Phi,\lambda,\) and \(a_0\). 

Taking a union bound and using the definition of \(\eps_\ell\), we obtain \[\Pr(\cE_{k,m}) \leq \sum_{r=0}^{m-1} \ee\left[\sum_{u\in L_r(\rho)}\mathbb1\{M_{u,k-r} \neq \tilde M_{u, k-r}\}\right] \leq \sum_{r=0}^{m-1}\ee[|L_r(\rho)|]\eps_{k-r} = \sum_{r=0}^{m-1} d^r\eps_{k-r}\] as desired.
\end{proof}

\begin{prop}\label{prop:sec-tree-E_k=tilde-E_k}
If \(\Phi, \lambda\), and \(\kappa\) satisfy Assumption~\ref{assump:noise-kernel}, then \[\lim_{k\to\infty}\ee\|\pi_{\rho,k}^\star - \tilde \pi_{\rho,k}^\star\| = 0.\] Consequently, \[\lim_{k\to\infty} (E_k^\star-\tilde E_k^\star) = 0.\]
\end{prop}
\begin{proof}
Let \(\cF_k = \sigma(T, \tau_{L_k(\rho)}^\star)\) and \(\tilde\cF_k = \sigma(T, \tilde\tau_{L_k(\rho)}^\star)\). We know that \(\tau_\rho^\star \to (T, \tau_{L_k(\rho)}^\star) \to \tilde\tau_{L_k(\rho)}^\star\) is a Markov chain. Hence, for any measurable \(A\), \begin{align*}
    \tilde\pi_{\rho,k}^\star(A) &= \Pr(\tau_\rho^\star\in A|\tilde \cF_k)\\
    &= \ee[\Pr(\tau_\rho^\star\in A|\sigma(\cF_k\cup \tilde \cF_k)) | \tilde\cF_k]\\
    &= \ee[\pi_{\rho,k}^\star(A) |\tilde\cF_k],
\end{align*} implying that \(\tilde\pi_{\rho,k}^\star = \ee[\pi_{\rho,k}^\star | \tilde\cF_k]\).

By the triangle inequality, \[\|\pi_{\rho,k}^\star-\tilde\pi_{\rho,k}^\star\| \leq \|\pi_{\rho,k}^\star-\tilde\pi_{\rho,k,m}^\star\| + \|\tilde\pi_{\rho,k,m}^\star-\tilde\pi_{\rho,k}^\star\|.\] Observing that \(\tilde\pi_{\rho,k,m}^\star\) is \(\tilde\cF_k\)-measurable, then the second term is \[\tilde\pi_{\rho,k,m}^\star - \tilde\pi_{\rho,k}^\star = \tilde\pi_{\rho,k,m}^\star-\ee[\pi_{\rho,k}^\star|\tilde\cF_k] = \ee[\tilde\pi_{\rho,k,m}^\star-\pi_{\rho,k}^\star|\tilde\cF_k].\] Convexity of the TV norm and conditional Jensen yields \[\|\tilde\pi_{\rho,k,m}^\star - \tilde\pi_{\rho,k}^\star\| =\left\|\ee[\tilde\pi_{\rho,k,m}^\star - \pi_{\rho,k} ^\star| \tilde\cF_k]\right\| \leq \ee[\|\tilde\pi_{\rho,k,m}^\star-\pi_{\rho,k}^\star\| | \tilde\cF_k].\] Taking expectation on both sides, \[\ee\|\tilde\pi_{\rho,k,m}^\star - \tilde\pi_{\rho,k}^\star\| \leq \ee\|\tilde\pi_{\rho,k,m}^\star-\pi_{\rho,k}^\star\|.\] Hence, \[\ee\|\pi_{\rho,k}^\star-\tilde\pi_{\rho,k}^\star\| \leq 2\ee\|\pi_{\rho,k}^\star-\tilde\pi_{\rho,k,m}^\star\|.\] Therefore, \begin{align*}
    \ee\|\pi_{\rho,k}^\star - \tilde\pi_{\rho,k}^\star\| &\leq 2\ee\|\pi_{\rho,k}^\star - \tilde\pi_{\rho,k,m}^\star\|\\
    [\text{triangle inequality}]&\leq 2\ee\|\pi_{\rho,k}^\star - \pi_{\rho,k,m}^\star\| + 2\ee\|\pi_{\rho,k,m}^\star - \tilde\pi_{\rho,k,m}^\star\|\\
    [\text{Lemmas~\ref{lem:E_k-S_km}, \ref{lem:u_km}, and~\ref{lem:S_km-tilde S_km}}]&\leq 2\alpha^m + 2C\sum_{r=0}^{m-1}d^r\gamma^{k-r}.
\end{align*}

Fix \(m\geq 1\). Then, \[\limsup_{k\to\infty} \ee\|\pi_{\rho,k}^\star - \tilde\pi_{\rho,k}^\star\| \leq 2\alpha^m.\] Taking \(m\to\infty\), we conclude the first desired equality.

Observe that \(E_k^\star = \ee[\sup_x\pi_{\rho,k}^\star(\{x\}) ]\) and \(\tilde E_k^\star = \ee[\sup_x\tilde\pi_{\rho,k}^\star(\{x\})]\). By Jensen, \[|E_k^\star-\tilde E_k^\star| \leq \ee\left|\sup_x\pi_{\rho,k}^\star(\{x\}) - \sup_x\tilde\pi_{\rho,k}^\star(\{x\})\right| \leq \ee\sup_x|\pi_{\rho,k}^\star(\{x\}) - \tilde\pi_{\rho,k}^\star(\{x\})| \leq \ee\|\pi_{\rho,k}^\star-\tilde\pi_{\rho,k}^\star\|.\] Since the noisy boundary labels are obtained by randomly perturbing the true boundary labels, then \(\tilde E_k^\star \leq E_k^\star\), meaning that we, in fact, have \[0 \leq E_k^\star-\tilde E_k^\star \leq\ee\|\pi_{\rho,k}^\star-\tilde\pi_{\rho,k}^\star\|.\] Taking the limit yields our desired claim.
\end{proof}


\section{Robust Reconstruction Under Constant Probability Label Retention}\label{subsec:easy-noise-matrix}
The previous section developed a quantitative pruning approximation for the continuum model. We now transfer this approximation to the finite-\(q\) Potts broadcast model. The only new obstruction is that, for finite \(q\), independently generated labels and erroneous noisy observations may coincide with positive probability. In this section, we prove robust reconstruction under constant probability label retention. In other words, \(a_q \geq a_0\) for some \(a_0\in(0,1)\), so that each true depth-\(k_q\) label is retained with probability at least \(a_0\). 

\begin{assumption}\label{assump:noise-matrix-easy}
Let \(\Phi\) be a PGF of an \(\nn_0\)-valued random variable with finite mean \(\Phi'(1) = d\). Let \(\lambda \in(0,1)\). Consider the sequences \((k_q)_{q\in\nn}\) of depth and \((\Delta^{(q)})_{q\in\nn}\) of noise matrices. We say that \(\Phi, \lambda, (k_q)\), and \((\Delta^{(q)})\) satisfy Assumption~\ref{assump:noise-matrix-easy} if \begin{enumerate}
    \item \(d\lambda > 1\),
    \item \(\Phi\) has finite second moment, 
    \item \(a_q \geq a_0\) for some constant \(a_0 \in (0,1]\),
    \item \(b_q \to 0\), and
    \item \(k_q \to \infty\)
\end{enumerate} as \(q\to\infty\).
\end{assumption}

We note that Assumption~\ref{assump:noise-matrix-easy} is a special case of Assumption~\ref{assump:noise-matrix} if we further assume it has finite moments of sufficiently high order. Indeed, \(a_q(d\lambda)^{k_q} \to\infty\), since \(a_q \geq a_0 > 0\) and \(d\lambda > 1\), and \(b_q^{2s-1}/a_q^{2s\zeta}\to 0\).

Recall that \(E_\infty^\star = \lim_{k\to\infty} E_k^\star\).

Throughout this section, starred quantities refer to the continuum model, while unstarred quantities refer to the finite-\(q\) model, unless stated otherwise. We suppress the superscript \((q)\) on finite-\(q\) quantities when clear from the context.

Let \(E_k := p_T(\Phi, \lambda, q, k) = \ee[\max_i X_{q,k}(i)]\) and \(\tilde E_k := \tilde p_T(\Phi, \lambda, q, k) = \ee[\max_i W_{q,k}(i)]\). Let \(\pi_{\rho, k} := X_{q,k}\) and \(\tilde\pi_{\rho, k} := W_{q,k}^\Delta\) denote the true posterior measures given the noiseless and noisy depth-\(k\) labels, respectively. Let \(\pi_{\rho, k, m}\) and \(\tilde\pi_{\rho, k, m}\) denote the auxiliary synthetic measures given the boundary labels produced by an analogous finite-\(q\) pruning construction and numerical recursion as in Definition~\ref{def:T_*} on the noiseless and noisy depth-\(k\) labels, respectively. In this section, we assume that the noisy model adheres to Assumption~\ref{assump:noise-matrix-easy}.

\begin{prop}\label{prop:tree-noise-matrix-easy}
If \(\Phi, \lambda, (k_q)\), and \((\Delta^{(q)})\) satisfy Assumption~\ref{assump:noise-matrix-easy}, then \[\ee\|X_{q, k_q} - W_{q, k_q}\|\to0,p_T(q) = E_{k_q}\to E_\infty^\star,\text{ and } \tilde p_T(q) = \tilde E_{k_q}\to E_\infty^\star\] as \(q\to\infty\).
\end{prop}

The main difficulty in transferring the continuum estimates to the finite-alphabet model is that independent redraws and erroneous observations may coincide with true labels with positive probability. We first couple the finite-\(q\) and continuum models and control accidental collisions. This allows us to transfer the continuum pruning estimates to a root neighborhood of depth \(R_q\to\infty\), with \(d^{2R_q}(q^{-1}+b_q)\to0\). The collision bounds below then show that this neighborhood contains no accidental collisions with probability tending to one, allowing us to apply the continuum pruning estimates there.

To handle arbitrary boundary depths, we then estimate the depth-\(R_q\) labels from the original observations. The key is to preserve a positive correct-output probability and vanishing probabilities of specified incorrect outputs throughout this reduction. We take the bottom \(k_q - R_q\), split it into sufficiently sized blocks, and recursively predict the top labels of each block using the top labels of the block below until we have a prediction of the depth-\(R_q\) labels. We show that these predictions can themselves be thought of as noisy labels generated according to a noise matrix that continues to satisfy Assumption~\ref{assump:noise-matrix-easy}. Consequently, whenever \(a_q\ge a_0 > 0\) and \(b_q\to0\), the noisy and noiseless finite-\(q\) root posteriors are asymptotically equivalent in expected total variation.

For two conditionally independent noisy labels with the same true input \(i\), \[\Pr(\tilde\tau_u=\tilde\tau_v\neq i | \tau_u = \tau_v = i) = \sum_{j\neq i} \Delta_{ij}^2 \leq b_q\sum_{j\neq i}\Delta_{ij} \leq b_q.\] For distinct true inputs \(i\neq i'\), \[\Pr(\tilde\tau_u=\tilde\tau_v | \tau_u = i, \tau_v = i') = \sum_j \Delta_{ij}\Delta_{i'j} \leq \Delta_{ii}\Delta_{i'i} + \sum_{j\neq i}\Delta_{ij}\Delta_{i'j} \leq 2b_q.\] We know that in the Potts broadcast model, mutation labels collide with probability \(1/q\).

Let \[L_{\leq k}(u) = \bigcup_{r=0}^k L_r(u).\] Suppose \(N_k := |L_{\leq k}(\rho)| \leq M\). Then there are at most \(M\) vertices at which a fresh label is drawn, namely the root and the vertices reached across mutation edges, and these fresh label draws are independent. There are also at most \(M\) noisy boundary labels. 

We now couple the \(q\)-state Potts broadcast process to the continuous broadcast process by using the same copy-or-mutation decisions on every edge and the same \(\mu\)-distributed draws at the root and at every mutation. Define \[T_q(x):=\min\{q,1+\lfloor qx\rfloor\}\] for \(x\in[0,1]\) and \[I_j := T_q^{-1}(\{j\})\] for \(j\in[q]\), and set \[\tau_u=T_q(\tau_u^\star)\] for every \(u\in T\). Since \(T_q(Y)\sim\mathrm{Unif}([q])\) whenever \(Y\sim\mu\), this gives the \(q\)-state Potts broadcast process. 

Let \(W_k^0\) be the event that two distinct fresh continuum labels drawn at the root or at mutation vertices in \(L_{\leq k}(\rho)\) are mapped by \(T_q\) to the same element of \([q]\). Then, \[\Pr(W_k^0 | T_{\leq k}) \leq \frac{N_k^2}{2q}.\] On \((W_k^0)^c\), we have \(\tau_u=\tau_v\) if and only if \(\tau_u^\star=\tau_v^\star\) for all \(u,v\in L_{\le k}(\rho)\). Hence, on this event \((W_k^0)^c\), the equality pattern of the \(q\)-state labels agrees exactly with that of the continuous broadcast process.

For \(u\in L_k(\rho)\), let \(\tilde\tau_u\) denote the actual finite-\(q\) noisy observation, so that, conditional on \(\tau_u=i\), \[\Pr\left(\tilde\tau_u=j\big|\tau_u=i\right)=\Delta_{ij}.\] We couple this finite-\(q\) noisy observation to an auxiliary continuum noisy observation \(\tilde\tau_u^\star\) by setting \[\tilde\tau_u^\star= \begin{cases}
    \tau_u^\star, & \text{if }\tilde\tau_u=\tau_u,\\
    \tilde Y_{u,j}, & \text{if }\tilde\tau_u=j\neq\tau_u,
\end{cases}\] where \(\tilde Y_{u,j}\sim\mathrm{Unif}(I_j)\) are independent over \(u\). Conditional on \(\tau_u^\star=x\), the auxiliary observation \(\tilde\tau_u^\star\) equals \(x\) with probability \(\Delta_{T_q(x),T_q(x)}\). Thus, it is an auxiliary noisy kernel of the form in Assumption~\ref{assump:noise-kernel}, with \[a_q(x) = \Delta_{T_q(x),T_q(x)} \geq a_0.\] In particular, for any \(x\in I_i\) and Borel \(B\subset[0,1]\), \[\kappa(x, B) = \Delta_{ii}\mathbb1\{x\in B\} + q\sum_{j\neq i}\Delta_{ij}\mu(B\cap I_j).\]

Let \(W_k\) be the union of \(W_k^0\), the event that an erroneous noisy boundary observation agrees with a correctly observed boundary label, and the event that two erroneous noisy boundary observations agree. This is the event of accidental collisions that we wish to avoid. On \(W_k^c\), the finite-\(q\) and continuum constructions have the same equality patterns for both the true labels and the noisy boundary observations. A union bound yields \[\Pr\left(W_k | T_{\leq k}\right)  \leq C_0N_k^2\left(\frac{1}{q}+b_q\right)\] for some constant \(C_0 > 0\).

\begin{lem}\label{lem:pi_rho,kq-tilde-pi_rho,kq}
Fix \(a_0\in(0,1]\) and suppose \(a_q \geq a_0\). There are constants \(C > 0\) and \(\gamma\in(0,1)\), depending only on \(\Phi, \lambda\), and \(a_0\), such that, for every \(m\in[k]\), \[\ee\|\pi_{\rho, k}-\tilde\pi_{\rho,k,m}\| \leq \alpha^m + C\sum_{r=0}^{m-1}d^r\gamma^{k-r} + Cd^{2k}\left(\frac{1}{q}+ b_q\right).\] Moreover, \[|E_k - E_k^\star| \leq Cd^{2k}/q + 1/q.\]
\end{lem}
\begin{proof}
Let \(N_k = |L_{\leq k}(\rho)|, Z_r^T = |L_r(\rho)|\), and \(\sigma_D^2 = \mathrm{Var}(D)\), where \(D\sim \Phi\). We know that the branching process gives \(\ee Z_r^T = d^r\) and \(\ee(Z_r^T)^2 \leq C_1d^{2r}\) (see, e.g., Chapter I Section 2 of~\cite{AN72}). Since \(d > 1\), Minkowski's inequality gives \((\ee N_k^2)^{1/2} \leq \sum_{r=0}^{k} (\ee(Z_r^T)^2)^{1/2} \leq C_2d^k.\) Hence, \[\Pr(W_k) = \ee[\Pr(W_k | N_k)] \leq \ee[C_0N_k^2(q^{-1}+b_q)] = C_3d^{2k}(q^{-1}+b_q).\]

Consider \(Q_{q, k} :=(T_q)_*\pi_{\rho,k}^\star\), the pushforward measure of \(\pi_{\rho, k}^\star\) by \(T_q\). We know that \(\pi_{\rho, k}^\star = \mathcal L(\tau_\rho^\star | T_{\leq k}, \tau_{L_k(\rho)}^\star)\), so \(Q_{q, k} = \mathcal L(T_q(\tau_\rho^\star) | T_{\leq k}, \tau_{L_k(\rho)}^\star) = \mathcal L(\tau_\rho| T_{\leq k}, \tau_{L_k(\rho)}^\star)\) under our coupling \(T_q(\tau_\rho^\star) = \tau_\rho\). We compare \(Q_{q, k}\) and \(\pi_{\rho,k}\), where \(\pi_{\rho,k} = \mathcal L(\tau_\rho | T_{\leq k}, \tau_{L_k(\rho)})\). Since the continuous boundary labels contain strictly more information than the discrete labels, the distributions are different. In fact, we know that \[\pi_{\rho,k} = \ee[Q_{q, k}|\cF_{q,k}],\] where \(\cF_{q,k} = \sigma(T_{\leq k}, \tau_{L_k(\rho)})\).

Now, let \(H_{q,k}\) be the probability vector on \([q]\) obtained by applying the continuum posterior recursion of Lemma~\ref{lem:posterior-measure-recursion} to the finite-\(q\) true boundary labels, treating each distinct label in \([q]\) as a distinct atom of the continuum model. That is, we initialize each boundary vertex with \(\delta_{\tau_v}\), recursively apply the same numerical update rule as in Lemma~\ref{lem:posterior-measure-recursion}, and identify the nonatomic component \(c_u\mu\) with \(c_u U_q = c_uT_q(\mu)\), where \(U_q\) is the uniform distribution on \([q]\). We break ties uniformly at random and independent of other assumptions. We use \(H_{q,k}\) to compare \(Q_{q, k}\) and \(\pi_{\rho,k}\). On the event of \((W_k^0)^c\), we must have \(H_{q,k} = Q_{q, k}\), since we have no accidental collisions leading to the same equality patterns and both posteriors treat the boundary labels in the same way. Thus, by the triangle inequality, the fact that \(H_{q,k}\) is \(\cF_{q,k}\)-measurable, and conditional Jensen, we have that \begin{align*}
    \ee\|Q_{q, k} - \pi_{\rho, k}\| &\leq  \ee\|\ee[Q_{q, k}|\cF_{q,k}] - H_{q,k}\| + \ee\|H_{q,k} - Q_{q, k}\|\\
    &=\ee\|\ee[Q_{q, k} - H_{q,k}|\cF_{q,k}]\| + \ee\|H_{q,k} - Q_{q, k}\|\\
    &\leq 2\ee\|Q_{q, k} - H_{q,k}\|\\
    &\leq 2\Pr(W_k^0).
\end{align*}

Let \(\tilde\pi_{\rho, k,m}^\star\) be the continuum noisy synthetic message for the auxiliary kernel above. Then \(\tilde\pi_{\rho,k,m} = (T_q)_*\tilde\pi_{\rho,k,m}^\star\) outside \(W_k\). Because the TV norm contracts under pushforward, then \begin{align*}
    \ee\|\pi_{\rho, k} - \tilde \pi_{\rho, k,m}\| &\leq \ee\|\pi_{\rho, k} - Q_{q, k}\| + \ee\|Q_{q, k} - (T_q)_*\tilde \pi_{\rho, k,m}^\star\| + \Pr(W_k)\\
    &\leq 2\Pr(W_k) + \ee\|\pi_{\rho, k}^\star- \tilde\pi_{\rho, k,m}^\star\| + \Pr(W_k)\\
    [\text{Lemmas~\ref{lem:E_k-S_km},~\ref{lem:u_km},~\ref{lem:S_km-tilde S_km}}]&\leq 3\Pr(W_k) + \alpha^m + C_4\sum_{r=0}^{m-1}d^r\gamma^{k-r}.
\end{align*} Using our earlier bound on \(\Pr(W_k)\) and choosing \(C\) appropriately yields our first desired inequality.

Write \(\pi_{\rho, k}^\star = c_\rho\mu + \sum_zp_\rho(z)\delta_z\). Recall that \(E_k^\star = \ee[\sup_z \pi_{\rho,k}^\star(\{z\})] = \ee[\sup_z p_\rho(z)]\). Outside \(W_k^0\), distinct atoms of \(\pi_{\rho, k}^\star\) are mapped by \(T_q\) to distinct elements of \([q]\). Hence, for every \(i\in[q]\), \[Q_{q, k}(i) = \frac{c_\rho}{q} + \sum_{z: T_q(z) = i} p_\rho(z),\] and on \((W_k^0)^c\), the sum on the right contains at most one nonzero atomic term. Thus, \[\ee\left|\max_iQ_{q, k}(i) - \sup_z\pi_{\rho, k}^\star(\{z\})\right| \leq \Pr(W_k^0) + 1/q\] on \((W_k^0)^c\). Indeed, if \(\pi_{\rho, k}^\star\) has an atom of maximal mass \(p\), then \(\max_i Q_{q, k}(i) = c_\rho/q + p\), while if \(\pi_{\rho, k}^\star\) has no atoms, then \(Q_{q, k} = U_q\). 

We know that \(|\max_i x_i - \max_i y_i| \leq \|x-y\|\) for any two probability vectors \(x, y\) because the maximum-coordinate functional on probability vectors is 1-Lipschitz with respect to the total variation norm. Therefore, \begin{align*}
    |E_k^\star - E_k| &= \left|\ee\max_i \pi_{\rho,k}(i) - \ee\sup_z\pi_{\rho, k}^\star(\{z\})\right|\\
    &\leq \ee\left|\max_i \pi_{\rho,k}(i) - \max_i Q_{q, k}(i)\right| + \ee\left|\max_i Q_{q, k}(i) - \sup_z\pi_{\rho, k}^\star(\{z\})\right|\\
    &\leq \ee\|\pi_{\rho,k} - Q_{q, k}\| + \Pr(W_k^0) + \frac{1}{q}\\
    &\leq 3\Pr(W_k^0) + \frac{1}{q}\\
    &= 3\ee[\Pr(W_k^0|N_k)] + \frac{1}{q}\\
    &\leq C\frac{d^{2k}}{q} + \frac{1}{q}
\end{align*} for appropriate \(C > 0\) as desired.
\end{proof}

Lemma~\ref{lem:pi_rho,kq-tilde-pi_rho,kq} applies only at depths where collisions remain unlikely. We now remove this restriction by repeatedly estimating labels from observations a fixed number \(h_0\) of generations below, using the following rule.

We reduce the depth of the tree from \(k_q\) to \(R_q\) by iteratively applying an \(h_0\)-decoder defined as follows. Each application of the \(h_0\)-decoder reduces the depth of the tree by \(h_0\).

\begin{defin}
Let \(\sigma\) be an arbitrary labeling with alphabet \([q]\) on \(T\subset\nn^*\). Then, for any \(u\in T\subset\nn^*\), define \[J_{u,k}^{(q)}(\sigma) := \{j\in[q]: \exists v\neq w\in L_k(u)\text{ s.t. }\sigma_v=\sigma_w=j\text{ and }\operatorname{LCA}(v,w) = u\}.\] Equivalently, $j \in J_{u,k}^{(q)}$ if and only if the label $j$ appears in at least two descendants at distance $k$ from $u$ belonging to distinct child-subtrees of $u$.
\end{defin}

\begin{defin}(\(h_0\)-decoder)
Let \(h_0\in \nn\). Then an application of the \(h_0\)-\emph{decoder} given a tree with an arbitrary depth-\(k\) labeling \(\tau_{L_k(\rho)}\) estimates the labels of depth-\((k-h_0)\) by \[\bar\tau_u := \begin{cases}
    i, &\text{if }J_{u,h_0}^{(q)} = \{i\}\\
    Y_u^{(q)}\sim\mathrm{Unif}([q]), &\text{otherwise}.
\end{cases}\] Here, the random variables \(Y_u^{(q)}\) are independent for every depth-\((k-h_0)\) vertex.
\end{defin}

\begin{lem}\label{lem:h_0-decoder}
For every \(a_0 \in(0,1]\), there exist constants \(a_*\in(0,a_0], b_* > 0,  h_0\in\nn\), and \(q_0\in\nn\), such that for every \(q\geq q_0\), if \[\min_i \Delta_{ii} \geq a_*\text{ and }\max_{i\neq j}\Delta_{ij} \leq b_*,\] where \(\Delta\) is the noise matrix for the depth-\(k\) labels, then applying the \(h_0\)-decoder to the depth-\(k\) labels produces a noise matrix \(\hat\Delta\) for the depth-\((k-h_0)\) labels satisfying \[\min_i\hat\Delta_{ii} \geq a_*\text{ and }\max_{i\neq j}\hat\Delta_{ij} \leq \frac{\max_{i\neq j}\Delta_{ij}}{2} + \frac{2}{q} \leq b_*.\]
\end{lem}
\begin{proof}
Set \(a_* := \min\{a_0/2, (1-\xi)(1-\alpha)/8\}\). Let \(u\in L_{k-h_0}(\rho)\). Consider only descendants connected to \(u\) by open copy edges. Each of their boundary labels is correct with probability at least \(a_*\). By an independent thinning process, we may reduce this probability to exactly \(a_*\), giving a lower bound to the two-witness event. By the computation in Lemma~\ref{lem:lim-of-two-witness}, the two-witness probability in a height-\(h\) block is \[g_h(a_*)=\omega(f^{\circ(h-1)}(1-a_*)) \to (1-\xi)(1-\alpha).\] 

Fix \(h_0\) sufficiently large such that \(g_{h_0}(a_*) \geq 4a_*\). Let \(b := \max_{r\neq s}\Delta_{rs}\). Fix \(i\in[q]\). Conditional on \(\tau_u = i\) and any realization of subtree below \(u\), the distribution of a boundary observation at distance \(h_0\) is the row \(i\) of \(M_q^{h_0}\Delta\). For \(j\neq i\), put \[p_j := (M_q^{h_0}\Delta)_{ij} = \lambda^{h_0} \Delta_{ij} + (1-\lambda^{h_0})\frac{1}{q}\sum_{\ell=1}^q\Delta_{\ell j} \leq b + \frac{1}{q}.\] We know that \(\sum_j p_j = 1\).

Two boundary observations in different first-child subtrees are independent conditional on the label of \(u\) and the tree. The expected number of unordered pairs of leaves in distinct first-child subtrees is \[C_{h_0} := \frac{1}{2}\ee\sum_{\substack{v, w\in L_1(u)\\ v\neq w}} |L_{h_0-1}(v)||L_{h_0-1}(w)| = \frac{\ee[D(D-1)]}{2}d^{2(h_0-1)} < \infty.\] A union bound gives \[\Pr(j\in J_{u, h_0}^{(q)} | \tau_u = i) \leq C_{h_0}p_j^2.\] Then \begin{align*}
    \Pr(\text{some wrong witness label}|\tau_u = i) &\leq C_{h_0}\sum_{j\neq i}p_j^2\\
    &\leq C_{h_0}(b+1/q)\sum_{j\neq i}p_j\\
    &\leq C_{h_0}(b+1/q).
\end{align*}

If a true witness exists and no wrong witness exists, the output is correct. Thus, \[\hat\Delta_{ii} \geq 4a_* - C_{h_0}(b+1/q).\] For \(j\neq i\), the decoder outputs \(j\) only if \(j\) has a witness or if its uniformly random label is \(j\). Hence, \[\hat\Delta_{ij} \leq C_{h_0}(b+1/q)^2 + 1/q.\] Choose \(b_* > 0\) sufficiently small so that \(C_{h_0}b_* \leq a_*\) and \(2C_{h_0}b_* \leq 1/2\). Then choose \(q_0\) sufficiently large so that \(C_{h_0}/q \leq a_*, 2C_{h_0}/q^2\leq 1/q\), and \(2/q \leq b_*/2\) for \(q \geq q_0\). Then \(\hat\Delta_{ii} \geq 4a_* - C_{h_0}(b+1/q) \geq 2a_*\) and \(\max_{j\neq i}\hat\Delta_{ij} \leq 1/q + 2C_{h_0}b^2 + 2C_{h_0}/q^2 \leq b/2 + 2/q \leq b_*\) as desired.
\end{proof}

We now prove Proposition~\ref{prop:tree-noise-matrix-easy}.
\begin{proof}[of Proposition~\ref{prop:tree-noise-matrix-easy}]
Let \(a_*, b_*, C, h_0\), and \(q_0\) be as in Lemma~\ref{lem:h_0-decoder}. For sufficiently large \(q\), the original noise matrix satisfies its hypothesis, since \(a_* \leq a_0\leq a_q\) and \(b_q \to 0\).

Set \(\Delta^{(q, 0)} = \Delta^{(q)}\). Recursively let \(\Delta^{(q, t+1)}\) be the output noise matrix of the \(h_0\)-decoder with boundary labels according to \(\Delta^{(q,t)}\). Write \(a_{q, t} := \min_i \Delta_{ii}^{(q,t)}\) and \(b_{q, t} := \max_{i\neq j} \Delta_{ij}^{(q,t)}\). By Lemma~\ref{lem:h_0-decoder}, \(a_{q,t} \geq a_*\) and \(b_{q, t+1} \leq \frac{1}{2}b_{q,t} + \frac{2}{q}\). Induction yields \(b_{q,t} \leq 2^{-t}b_q+\frac{4}{q}(1-2^{-t})\leq b_q + \frac{4}{q}\).

Set \(R_q^0 := \left\lfloor\min\left\{\frac{k_q}{2}, \frac{\log\frac{1}{b_q + 1/q}}{4\log d}\right\}\right\rfloor\), since we require \(R_q \leq k_q\) and \(R_q\) must increase sufficiently slowly compared to the number of labels and noise injected into the labels. For large \(q\), choose \(R_q \in\{R_q^0, \dots, R_q^0 + h_0-1\}\) satisfying \(R_q \equiv k_q\pmod{h_0}\). Then, indeed, we have that \(R_q\to\infty, R_q\leq k_q/2 + h_0 \leq k_q\), and \(d^{2R_q}(b_q + 1/q) \leq d^{2h_0}\sqrt{b_q + 1/q} \to 0\). Let \(t_q := (k_q-R_q)/h_0\) be the number of iterations we apply the decoder and \(m_q := \lfloor\sqrt{R_q}\rfloor\).

Apply the \(h_0\)-decoder recursively to the bottom \(t_q\) blocks. This produces a decoded label \(\hat\tau_u\) at every \(u\in L_{R_q}(\rho)\). Conditional on the true depth-\(R_q\) labels and \(T_{\leq R_q}\), these decoded labels are independent with noise matrix \(\Delta^{(q,t_q)}\), which satisfies \[\min_i\Delta_{ii}^{(q,t_q)} \geq a_*\text{ and }\max_{i\neq j}\Delta_{ij}^{(q,t_q)} \leq b_q + 4/q.\]

On the remaining top of the tree, \(T_{\leq R_q}\), let \(H_q\) be the finite-\(q\) synthetic message with pruning depth \(m_q\) and boundary labels \(\hat\tau_{L_{R_q}(\rho)}\). By Lemma~\ref{lem:pi_rho,kq-tilde-pi_rho,kq} with the noise matrix \(\Delta^{(q,t_q)}\), \[\eps_q := \ee\|\pi_{\rho, R_q} - H_q\| \leq \alpha^{m_q} + C\sum_{r=0}^{m_q-1}d^r\gamma^{R_q-r} + Cd^{2R_q}(b_q + 5/q).\] We note that \(\eps_q \to 0\), since all three terms tend to zero.

We now observe that \(H_q\) is measurable with respect to the original noisy boundary, the full observed tree, and the independent randomness in the decoder. By the Markov property of the broadcast process, conditional on \(T_{\leq R_q}\) and \(\tau_{L_{R_q}(\rho)}\), the descendant tree, descendant labels, boundary noise, and decoder randomness below generation \(R_q\) are independent of the root label. Let \(\cF_q = \sigma(T_{\leq k_q}, \tau_{L_{k_q}(\rho)}, \tilde\tau_{L_{k_q}(\rho)}, \mathcal R)\) and \(\tilde \cF_q = \sigma(T_{\leq k_q}, \tilde\tau_{L_{k_q}(\rho)}, \mathcal R)\), where \(\mathcal R\) is the randomness of the decoder. Then, \[\ee[\pi_{\rho,R_q}|\cF_q] = \pi_{\rho,k_q}\] and \[\ee[\pi_{\rho, R_q} | \tilde\cF_q] = \tilde\pi_{\rho, k_q}.\] Subtracting \(H_q\) and applying conditional Jensen to the half-\(\ell^1\) norm yields \[\ee\|\pi_{\rho, k_q} -H_q\| \leq \ee \|\pi_{\rho, R_q} - H_q\|\leq \eps_q\] and \[\ee\|\tilde\pi_{\rho, k_q} -H_q\| \leq \ee \|\pi_{\rho, R_q} - H_q\|\leq \eps_q.\] The triangle inequality then yields \[\ee\|\pi_{\rho, k_q} - \tilde\pi_{\rho, k_q}\| \leq 2\ee\|\pi_{\rho, R_q}-H_q\| = 2\eps_q \to 0\] as desired. Here, we remark that \(\pi_{\rho, k_q} = X_{q, k_q}\) and \(\tilde\pi_{\rho, k_q} = W_{q,k_q}\) by definition.

Taking the maximum coordinate yields \(|E_{k_q} - E_{R_q}| \leq 2\eps_q\) and \(|\tilde E_{k_q} - E_{R_q}| \leq 2\eps_q\). Then, by Lemma~\ref{lem:pi_rho,kq-tilde-pi_rho,kq} and the fact that \(d^{2R_q}(b_q+1/q)\to0\), we have that \(|E_{R_q} - E_{R_q}^\star| \to 0\). Since \(E_{R_q}^\star \to E_\infty^\star\), we conclude the desired limit.
\end{proof}

\section{Bootstrapping to Vanishing Label Retention}\label{subsec:bootstrap}
In this section, we prove the main theorem, in which \(a_q\) may vanish, i.e., the probability of retaining the true label vanishes. Throughout this section, we assume that \(\Phi, \lambda, (k_q)\), and \(\Delta^{(q)}\) satisfy Assumption~\ref{assump:noise-matrix}. 

The main idea is to construct decoded labels at an intermediate generation depth of \(k_q-h_q\) using the depth \(k_q\) boundary labels whose effective noise matrix has a fixed positive diagonal lower bound and vanishing off-diagonal entries. We can then apply Proposition~\ref{prop:tree-noise-matrix-easy} to that noise matrix on the remaining upper tree.

We expect approximately \((d\lambda)^h\) descendants of \(u\) at \(h\)-levels below to carry the label \(\tau_u\) as their true labels. After injecting the noise, we observe the true label with probability \(a_q\), so around \(a_q(d\lambda)^h\) are observed. Naturally, we require \(a_q(d\lambda)^h = \Omega(1)\). Hence, for this amplification step, we choose \(h\) so that \(a_q(d\lambda)^h\) is bounded below by a sufficiently large constant. Equivalently, we need \[h_q = \frac{\log\frac{1}{a_q}}{\log(d\lambda)} + H\] for some \(H > 0\), implying that \(h_q \asymp \frac{\log\frac{1}{a_q}}{\log(d\lambda)}\). Thus, \[d^{h_q}\asymp a_q^{-\zeta},\] where \(\zeta := \frac{\log d}{\log(d\lambda)}.\)

Since we have \(b_qd^{h_q} = \frac{b_q}{a_q^{\zeta}}\) noisy signals for each label, so we require \[b_q \ll a_q^\zeta,\] which motivates our requirement that \(\frac{b_q^{2s-1}}{a_q^{2s\zeta}}\to 0\) for some integer \(s > \zeta\) with \(\ee D^s < \infty\). For the remainder of the section, fix an integer \(s > \zeta\) such that \(\ee D^s < \infty\).

We first establish two moment estimates. Recall that \(Z_t^T := |L_t(\rho)|\). We continue to use \(Z_{\rho, t}\) for the number of depth-\(t\) vertices in the root's open copy component. By the copy component, we refer to the connected subgraph induced by all vertices connected to the root by a path consisting only of copy edges. Under the continuum coupling, this is the variable from Definition~\ref{defin:Z_rhok}. In the finite-\(q\) model, it need not count every vertex whose label happens to be equal to the root label, since we may have accidental collisions.

\begin{lem}\label{lem:moment-estimates}
For every integer \(\ell\in[s]\), there exists \(C_\ell > 0\) such that, for every \(t \geq 0\), \[\ee[(Z_t^T)^\ell] \leq C_\ell d^{\ell t}\text{ and }\ee[Z_{\rho, t}^\ell] \leq C_\ell(d\lambda)^{\ell t}.\]
\end{lem}
\begin{proof}
We first consider \(Z_t^T\). We use strong induction. For \(\ell = 1\), the claim follows from \(\ee Z_t^T = d^t\). Suppose the bounds hold for all positive integers smaller than \(\ell\). Conditional on \(Z_t^T = z\), the next generation is the sum of \(z\) independent copies of \(D\). Expanding the \(\ell\)th power and grouping terms according to the number of distinct summands gives \[\ee[(Z_{t+1}^T)^\ell | Z_t^T=z] \leq d^\ell z^\ell + C_\ell \sum_{j=1}^{\ell-1}z^j.\] Indeed, the terms with \(\ell\) distinct summands contribute \(d^\ell(z)_\ell \leq d^\ell z^\ell\) and all remaining terms involve moments of \(D\) of order at most \(\ell\). Here, \((z)_\ell = \frac{z!}{(z-\ell)!}\) refers to the falling factorial.

Taking expectations and applying the inductive hypothesis yields \[\ee[(Z_{t+1}^T)^\ell] \leq d^\ell\ee[(Z_t^T)^\ell] + C_\ell d^{(\ell-1)t}.\] Dividing by \(d^{\ell(t+1)}\) and summing over \(t\) proves the desired bound.

For the copy process, the offspring variable, conditional on \(D\), is \(\mathrm{Bin}(D, \lambda)\). Its mean is \(d\lambda > 1\), and its moments up to order \(s\) are finite because it is bounded by \(D\). The same argument, with \(d\) replaced by \(d\lambda\), proves the second inequality.
\end{proof}

\begin{lem}\label{lem:inequality}
Let \(z_1,\dots, z_n \geq 0\), and let \(B_1,\dots, B_n\) be independent \(\mathrm{Bernoulli}(1/q)\) random variables. Then, \[\ee\left[\left(\sum_{\ell=1}^nz_\ell B_\ell\right)^s\right] \leq C_s\left[\left(\frac{1}{q}\sum_{\ell=1}^nz_\ell\right)^s + \frac{1}{q}\sum_{\ell=1}^n z_\ell^s\right],\] where \(C_s\) depends only on \(s\).
\end{lem}
\begin{proof}
The claim is immediate if the weights are all zero. Otherwise, put \(S_r := \frac{1}{q}\sum_{\ell=1}^n z_\ell^r\) for each \(r\in[s]\). Expanding the \(s\)th power and grouping terms according to the number \(j\) of distinct indices, independence of the \(B_\ell\) gives \[\ee\left[\left(\sum_{\ell=1}^nz_\ell B_\ell\right)^s\right] \leq C_s\sum_{j=1}^s\sum_{\substack{r_1+\cdots+r_j =s\\r_1,\dots, r_j\geq 1}}\prod_{\ell=1}^j S_{r_\ell}.\] Here, removing the restriction that the indices be distinct only increases the corresponding nonnegative sums. By H\H{o}lder's inequality, \[S_r \leq S_1^{(s-r)/(s-1)}S_s^{(r-1)/(s-1)}.\] Consequently, for \(r_1 +\cdots+ r_j = s,\) \[\prod_{\ell=1}^j S_{r_\ell} \leq (S_1^s)^{(j-1)/(s-1)} S_s^{(s-j)/(s-1)} \leq S_1^s + S_s.\] There are only finitely many terms in the expansion, with their number depending only on \(s\). Absorbing that number into \(C_s\) proves the claim.
\end{proof}

We now describe the reconstruction rule. Fix \(q\) and a height \(h \in\nn\). First, we thin the observations on the boundary. Independently at each observed boundary vertex \(v\), retain the noisy label \(j = \tilde\tau_v\) with probability \(a_q/\Delta_{jj}\) and otherwise replace it by an erasure symbol \(\dagger \notin[q]\). Denote the resulting label by \(\check\tau_v\). This operation uses only the observed label and the noise matrix, not the true label. Conditional on \(\tau_v = i\), it satisfies \[\Pr(\check\tau_v = j |\tau_v=i) = \frac{a_q\Delta_{ij}}{\Delta_{jj}}.\] In particular, \[\Pr(\check\tau_v = i |\tau_v=i) = a_q\text{ and }\Pr(\check\tau_v = j |\tau_v=i) \leq b_q\] for \(j\neq i\). The thinned observations remain conditionally independent given the true labels. Since the conditional probability of a correctly retained label is exactly \(a_q\) for every true input, then \(\mathbb1\{\check\tau_v=\tau_v\}\) are \(\mathrm{Bernoulli}(a_q)\) variables, independent of the tree and the broadcast.

We now define a high-multiplicity version of the two-witness rule.

\begin{defin}\label{defin:high-mult-two-witness}
Let \(\sigma\) be a labeling of \(L_h(u)\) with values in \([q]\cup\{\dagger\}\). For \(v\in L_1(u)\) and \(j\in[q]\), put \[N_{v,h}(j;\sigma) := \sum_{w\in L_{h-1}(v)}\mathbb1\{\sigma_w = j\}.\] Define \[J_{u,h}^{(q,s)}(\sigma) := \left\{j\in[q]: \exists v\neq w\in L_1(u)\text{ such that }N_{v,h}(j;\sigma) \geq 2s\text{ and }N_{w,h}(j;\sigma)\geq 2s\right\}.\] Thus, a label qualifies if it appears at least \(2s\) times in each of two distinct child-subtrees of \(u\). Erased observations are ignored.

Apply this rule to the thinned labels \(\check\tau\) and output \[\hat\tau_u := \begin{cases}
    j, &\text{if }J_{u,h}^{(q,s)}(\check\tau) = \{j\},\\
    Y_u^{(q)}, \text{otherwise},
\end{cases}\] where \(Y_u^{(q)}\sim\mathrm{Unif}([q])\) are independent.
\end{defin}

We bound the probability that a wrong label meets this requirement. There are two possible sources of such occurrences: vertices whose true label is the wrong candidate and vertices whose observed labels are incorrectly changed to that candidate by the noise.

Let \(s  > \zeta\) be an integer such that \(\ee D^s < \infty\) as in Assumption~\ref{assump:noise-matrix}.

\begin{lem}\label{lem:high-mult-two-witness}
For every \(h\in\nn\) and \(i\in[q]\), \[\Pr(J_{\rho,h}^{(q,s)}(\check\tau)\setminus\{i\} \neq\emptyset | \tau_\rho = i) \leq C_s\left[\frac{(a_q(d\lambda)^h)^{2s}}{q} + \frac{(a_qd^h)^{2s}}{q^{2s-1}} + d^{2sh}b_q^{2s-1}\right].\]
\end{lem}
\begin{proof}
Fix $i\in[q]$, and write \[\Pr_i(\cdot):=\Pr(\cdot| \tau_\rho=i)\text{ and }\mathbb E_i[\cdot]:=\mathbb E[\cdot|\tau_\rho=i].\] Consider one child \(v\in L_1(\rho)\) and put \(N := |L_{h-1}(v)|\). The estimates for this branch may equivalently be obtained conditional on any specified positive value of \(D_\rho\) and any one of its children. They do not depend on that value or the chosen child. For a wrong candidate \(j\neq i\), let \[C_{v,j} := \sum_{w\in L_{h-1}(v)}\mathbb1\{\tau_w = \check\tau_w = j\}\] and \[B_{v, j} := \sum_{w\in L_{h-1}(v)} \mathbb1\{\tau_w \neq j, \check\tau_w = j\}.\] Then \(N_{v,h}(j; \check\tau) = C_{v,j} + B_{v,j}\). 

Let \[T_{v,j} := \sum_{w\in L_{h-1}(v)} \mathbb1\{\tau_w=j\}.\] Expose the tree and the copy-or-mutation decisions, but not the fresh labels. Let \(z_1,\dots, z_n\) be the numbers of boundary vertices in the nonroot copy components that meet this branch. The root copy component has label \(i\) and therefore contributes nothing to \(T_{v,j}\). Every other component receives an independent uniform label in \([q]\). Consequently, conditional on this forest, \[T_{v,j}\stackrel{d}{=} \sum_{\ell}^nz_\ell B_\ell,\] where \(B_\ell\stackrel{\mathrm{iid}}{\sim} \mathrm{Bernoulli}(1/q)\). This representation allows different components to receive the same label.

We have that \(\sum_\ell z_\ell \leq N\). To bound \(\sum_\ell z_\ell^s\), assign each component to its highest vertex. At depth \(t\) from \(\rho\), the branch has \(d^{t-1}\) vertices in expectation. By Lemma~\ref{lem:moment-estimates}, for each such vertex, its number of open descendants at level \(h\) has \(s\)th moment at most \(C_s(d\lambda)^{s(h-t)}\). Counting all such vertices, including those that are not component roots, only increases the sum. Hence, \[\ee\sum_{\ell=1}^n z_\ell^s \leq C_s\sum_{t=1}^hd^{t-1}(d\lambda)^{s(h-t)} \leq C_s(d\lambda)^{sh},\] where the last inequality uses \((d\lambda)^s > d\). 

Applying Lemma~\ref{lem:inequality} conditional on the forest, and then Lemma~\ref{lem:moment-estimates}, gives \[\ee_iT_{v,j}^s \leq C_s\left(\frac{d^{sh}}{q^s} + \frac{(d\lambda)^{sh}}{q}\right).\] Conditional on the true labels, \(C_{v,j}\) is an independent thinning of the \(T_{v,j}\) true-\(j\) vertices with retention probability \(a_q\). Using falling factorials, we conclude that \begin{align*}
    \Pr_i(C_{v,j}\geq s) &\leq \frac{\ee_i[(C_{v,j})_s]}{s!}\\
    &= \frac{a_q^s\ee_i[(T_{v,j})_s]}{s!}\\
    &\leq C_s\left[\left(\frac{a_qd^h}{q}\right)^s + \frac{(a_q(d\lambda)^h)^s}{q}\right].
\end{align*} Squaring and summing \(j\neq i\) yields \[\sum_{j\neq i}\Pr_i(C_{v,j}\geq s)^2 \leq C_s\left[q\left(\frac{a_qd^h}{q}\right)^{2s} + \frac{(a_q(d\lambda)^h)^{2s}}{q}\right].\]

For \(w\in L_{h-1}(v)\), write \[e_{w,j} := \mathbb1\{\tau_w \neq j\}\frac{a_q\Delta_{\tau_w,j}}{\Delta_{jj}}.\] Conditional on the true labels, this is the probability that \(w\) contributes to \(B_{v,j}\). We know that \(0\leq e_{w,j} \leq b_q\) and \(\sum_{j=1}^q e_{w,j} \leq 1\). Independence of the boundary noises and thinning gives \[\ee[(B_{v,j})_s | T_{\leq h}, \tau_{L_{\leq h}(\rho)}] = \sum_{\substack{w_1,\dots, w_s\in L_{h-1}(v)}\\\text{pairwise distinct}}\prod_{\ell=1}^s e_{w_\ell, j} \leq b_q^{s-1} N^{s-1} \sum_{w\in L_{h-1}(v)} e_{w,j}.\]

Set \(p_{ij} := \sum_{\ell\neq j}(M^h)_{i\ell}\frac{a_q\Delta_{\ell j}}{\Delta_{jj}}\). Since every boundary vertex has distance \(h\) from \(\rho\), conditional on the tree and \(\tau_\rho=i\), its expected value of \(e_{w,j}\) is \(p_{ij}\). Moreover, \(0\leq p_{ij} \leq b_q\) and \(\sum_{j=1}^q p_{ij} \leq 1\). It follows that \[\Pr_i(B_{v,j}\geq s) \leq \frac{b_q^{s-1}p_{ij}\ee[N^s]}{s!}\leq C_s d^{sh}b_q^{s-1}p_{ij}.\] Thus, \[\sum_{j\neq i} \Pr_i(B_{v,j} \geq s)^2 \leq C_sd^{2sh}b_q^{2(s-1)}\sum_{j\neq i}p_{ij}^2 \leq C_s d^{2sh}b_q^{2s-1}.\]

Put \(u_{ij} := \Pr_i(N_{v,h}(j; \check\tau)\geq 2s)\). The event inside this probability implies \(C_{v,j} \geq s\) or \(B_{v,j} \geq s\). By \((x+y)^2 \leq 2x^2 + 2y^2\) and our earlier bounds, \[\sum_{j\neq i}u_{ij}^2 \leq C_s\left[\frac{(a_q(d\lambda)^h)^{2s}}{q} + \frac{(a_qd^h)^{2s}}{q^{2s-1}} + d^{2sh}b_q^{2s-1}\right].\] Conditional on \(D_\rho\) and the root label, different child-subtrees and their observations are independent. For any pair of distinct children, the probability that both branches have at least \(2s\) occurrences of a fixed \(j\) is \(u_{ij}^2\). Hence, a union bound over child pairs and wrong labels gives \[\Pr_i(J_{\rho,h}^{(q,s)}(\check\tau)\setminus\{i\}\neq \emptyset) \leq \ee\binom{D}{2}\sum_{j\neq i}u_{ij}^2.\] Since \(\ee D^2 < \infty\), absorbing \(\ee\binom{D}{2}\) into the constant completes the proof.
\end{proof}

We now choose the height of the reconstruction block in accordance with our observation at the start of this section. 

\begin{lem}\label{lem:h_q}
There exist an integer \(H \geq 1\) and a constant \(c_1 > 0\), depending only on \(\Phi, \lambda\), and \(s\) such that, with \(h_q := \left\lceil\frac{\log(1/a_q)}{\log(d\lambda)}\right\rceil + H\), we have, uniformly over \(i\in[q]\), \[\Pr(i\in J_{\rho,h_q}^{(q,s)}(\check\tau)|\tau_\rho = i) \geq c_1.\]
\end{lem}
\begin{proof}
Choose \(H\) sufficiently large so that \((d\lambda)^{H-1} \geq 4s\). Condition on \(D_\rho \geq 2\) and \(\tau_\rho=i\). For each child \(v\), let \(V_v\) count the descendants at distance \(h_q-1\) from \(v\) that are connected to \(v\) by copy edges and whose labels are correctly retained after thinning. 

Then, \[V_v|Z_{v,h_q-1}\sim \mathrm{Bin}(Z_{v,h_q-1}, a_q).\] Write \(\mu_q := \ee V_v = a_q(d\lambda)^{h_q-1}\). By the choice of \(h_q,\) \(\mu_q\geq (d\lambda)^{H-1} \geq 4s\). Then, Lemma~\ref{lem:moment-estimates} gives \[\ee V_v^2 = a_q(1-a_q)\ee Z_{v,h_q-1} + a_q^2 \ee Z_{v, h_q-1}^2 \leq \mu_q + C_2\mu_q^2 \leq (1+C_2)\mu_q^2.\] By Cauchy-Schwarz, \(\mu_q \leq \frac{\mu_q}{2} + \sqrt{\ee V_v^2\Pr(V_v \geq \mu_q/2)}\), implying that \(\Pr(V_v \geq 2s) \geq \Pr(V_v\geq \mu_q/2) \geq \frac{1}{4(1+C_2)}\). 

Call a child \(v\) successful if its edge from the root is open and \(V_v \geq 2s\). Under the conditioning above, these events are independent, each with probability at least \(p := \frac{\lambda}{4(1+C_2)}\). Their number stochastically dominates \(\mathrm{Bin}(D_\rho, p)\), so the probability of at least two successful children is at least \(p^2\). Thus, the true root label qualifies with probability at least \(\Pr(i\in J_{\rho, h_q}^{(q,s)}(\check\tau) | \tau_\rho = i) \geq c_1 := \Pr(D\geq 2)p^2> 0\). This lower bound is uniform in the root label and proves the claim.
\end{proof}

\begin{proof}[of Theorem~\ref{thm:tree-theorem}]
Fix an integer \(s > \zeta\) such that \(\ee D^s< \infty\) and \(b_q^{2s-1}/a_q^{2s\zeta} \to 0\). Such an integer exists by Assumption~\ref{assump:noise-matrix}. Let \(H\) be as in Lemma~\ref{lem:h_q}, and set \(h_q := \left\lceil \frac{\log(1/a_q)}{\log(d\lambda)}\right\rceil + H\).

We first show that applying the height-\(h_q\) reconstruction rule of Definition~\ref{defin:high-mult-two-witness} produces an effective noise matrix with a uniformly positive diagonal and vanishing off-diagonal entries.

The choice of \(h_q\) gives \((d\lambda)^H \leq a_q(d\lambda)^{h_q}\leq (d\lambda)^{H+1},\) and \(d^{h_q}\leq d^{H+1}a_q^{-\zeta}\). Hence, by Lemma~\ref{lem:high-mult-two-witness}, uniformly over \(i\in[q]\), \[\Pr(J_{\rho, h_q}^{(q,s)}(\check\tau)\setminus\{i\}\neq \emptyset | \tau_\rho = i) \leq C\left[\frac{1}{q} + \frac{b_q^{2s-1}}{a_q^{2s\zeta}} + \frac{a_q^{2s(1-\zeta)}}{q^{2s-1}}\right].\] We show that the final term also tends to zero under Assumption~\ref{assump:noise-matrix}.

Since \(a_q = \min_i \Delta_{ii}\), there exists some \(i\in[q]\) such that \(\Delta_{ii} = a_q\). The off-diagonal entries in this row sum to \(1-a_q\), and thus, \(b_q \geq \frac{1-a_q}{q-1}\). If \(a_q \leq 1/2\), then \(1/q \leq 2b_q\), implying that \[\frac{a_q^{2s(1-\zeta)}}{q^{2s-1}}\leq 2^{2s-1}a_q^{2s(1-\zeta)}b_q^{2s-1} = 2^{2s-1}a_q^{2s}\frac{b_q^{2s-1}}{a_q^{2s\zeta}} \leq 2^{2s-1}\frac{b_q^{2s-1}}{a_q^{2s\zeta}}.\] On the other hand, if \(a_q > 1/2\), then \[\frac{a_q^{2s(1-\zeta)}}{q^{2s-1}}\leq \frac{2^{2s(\zeta-1)}}{q^{2s-1}}\leq \frac{C}{q}.\] Consequently, \[\sup_{i\in[q]}\Pr(J_{\rho,h_q}^{(q,s)}(\check\tau)\setminus\{i\} \neq \emptyset | \tau_\rho = i) \leq C\left(\frac{1}{q} + \frac{b_q^{2s-1}}{a_q^{2s\zeta}}\right) = o(1).\]

Let \(\hat\Delta\) denote the effective noise matrix produced by the height-\(h_q\) reconstruction rule, i.e., \[\hat\Delta_{ij} := \Pr(\hat\tau_\rho = j|\tau_\rho = i).\] By Lemma~\ref{lem:h_q}, the true root label qualifies with probability at least \(c_1 > 0\), uniformly over \(i\in[q]\). If the true label qualifies and no wrong label qualifies, then the reconstruction rule outputs the true label. Hence, \[\min_i\hat\Delta_{ii} \geq c_1 - C\left(\frac{1}{q} + \frac{b_q^{2s-1}}{a_q^{2s\zeta}}\right) \geq \frac{c_1}{2}\] for all sufficiently large \(q\).

For \(j\neq i\), the reconstruction rule can output \(j\) only if \(j\) qualifies or if the independent uniform fallback label is equal to \(j\). Thus, \[\hat\Delta_{ij} \leq C\left(\frac{1}{q}+\frac{b_q^{2s-1}}{a_q^{2s\zeta}}\right) + \frac{1}{q}.\] Hence, letting \(\hat a_q := \min_i \hat\Delta_{ii}\) and \(\hat b_q := \max_{i\neq j}\hat\Delta_{ij}\), we have that \(\hat a_q \geq c_1/2 > 0\) and \(\hat b_q \to 0\).

We now determine the generation at which these reconstructed labels are produced. Set \(r_q := k_q - h_q\). By the definition of \(h_q\), \(r_q \geq k_q - \frac{\log(1/a_q)}{\log(d\lambda)} - H -1 = \frac{\log(a_q(d\lambda)^{k_q})}{\log(d\lambda)} - H - 1\). Assumption~\ref{assump:noise-matrix} gives \(a_q(d\lambda)^{k_q}\to\infty\), so \(r_q\to\infty\).

Apply the height-\(h_q\) reconstruction rule independently below every vertex \(u\in L_{r_q}(\rho)\), using the original noisy labels \(\tilde\tau_{L_{k_q}(\rho)}\). Denote the resulting reconstructed labels by \(\hat\tau_{L_{r_q}(\rho)}\). Conditional on \(T_{\leq r_q}\) and the true labels \(\tau_{L_{r_q}(\rho)}\), the descendant subtrees rooted at vertices in \(L_{r_q(\rho)}\) are independent. Thus, \[\Pr(\hat\tau_u = j_u,\forall u\in L_{r_q}(\rho) | T_{\leq r_q}, \tau_{L_{r_q}(\rho)}) = \prod_{u\in L_{r_q}(\rho)}\hat\Delta_{\tau_u, j_u}.\] Hence, the reconstructed depth-\(r_q\) labels are exactly boundary observations generated through the effective noise matrix \(\hat\Delta\).

Since \(s > \zeta > 1\), we have \(s\geq 2\), meaning \(\ee D^2 < \infty\). Thus, we have that \(\Phi, \lambda, (r_q)\), and \(\hat\Delta\)  satisfy Assumption~\ref{assump:noise-matrix-easy}. Defining the posterior using the reconstructed depth-\(r_q\) labels by \(\hat W_{q, r_q}(i) := \Pr(\tau_\rho = i | T_{\leq r_q}, \hat\tau_{L_{r_q}(\rho)})\), then Proposition~\ref{prop:tree-noise-matrix-easy} yields \[\eps_q := \ee\|X_{q,r_q} - \hat W_{q, r_q}\| \to 0.\]

It remains to transfer this estimate from depth \(r_q\) back to the original depth \(k_q\). Let \(\mathcal R_q\) denote all auxiliary randomness used in the thinning and fallback steps. Define \(\cF_q := \sigma(T_{\leq k_q}, \tau_{L_{k_q}(\rho)}, \tilde\tau_{L_{k_q}(\rho)}, \mathcal R_q)\) and \(\tilde\cF_q := \sigma(T_{\leq k_q}, \tilde\tau_{L_{k_q}(\rho)}, \mathcal R_q)\). Then \(\hat W_{q, r_q}\) is \(\tilde F_q\)-measurable, and hence, also \(\cF_q\)-measurable.

By the Markov property of the broadcast process, once \(\tau_{L_{r_q}(\rho)}\) and \(T_{\leq r_q}\) are given, everything below generation \(r_q\) is conditionally independent of the root label. Thus, \[\ee[X_{q,r_q}|\cF_q] = X_{q, k_q}.\] Indeed, after the true depth-\(k_q\) labels are revealed, the noisy labels and the auxiliary randomization contain no further information about the root. Similarly, \[\ee[X_{q,r_q} | \tilde\cF_q] = W_{q, k_q}.\] 

Subtracting \(\hat W_{q,r_q}\) and applying conditional Jensen to the TV norm gives \[\ee\|X_{q, k_q} - \hat W_{q, r_q}\| = \ee\|\ee[X_{q, r_q} - \hat W_{q,r_q}|\cF_q]\| \leq \ee\|X_{q, r_q} - \hat W_{q,r_q}\| = \eps_q\] and \[\ee\|W_{q, k_q} - \hat W_{q, r_q}\| = \ee\|\ee[X_{q, r_q} - \hat W_{q,r_q}|\tilde\cF_q]\| \leq \ee\|X_{q, r_q} - \hat W_{q,r_q}\| = \eps_q.\] By the triangle inequality, \[\ee\|X_{q, k_q} - W_{q, k_q}\| \leq 2\eps_q \to 0\] as desired.

Finally, the maximum-coordinate functional is 1-Lipschitz with respect to the TV norm. Since the noisy boundary is a random perturbation of the true boundary, \[0 \leq p_T(q) - \tilde p_T(q) \leq \ee\|X_{q, k_q} - W_{q,k_q}\| \to 0.\] Applying Proposition~\ref{prop:tree-noise-matrix-easy} with the identity noise matrix gives \(p_T(q)\to E^\star_\infty\). The preceding bound then implies \(\tilde p_T(q)\to E^\star_\infty\), completing the proof.
\end{proof}

\bibliographystyle{plain} 
\bibliography{refs}

\end{document}